\documentclass[11pt]{article}
\usepackage{graphicx, color, epsfig, verbatim, subfigure, epsf} 
\usepackage{amsmath}
\usepackage{amssymb}
\usepackage{amsfonts}
\usepackage{amsthm}
\usepackage{latexsym}
\usepackage{fancyhdr,lastpage}
\usepackage{xspace}
\usepackage{setspace}
\usepackage[affil-it]{authblk}

\newtheorem{theorem}{Theorem}
\newtheorem{lemma}[theorem]{Lemma}
\newtheorem{corollary}[theorem]{Corollary}
\newtheorem{Pro}[theorem]{Proposition}

\newtheorem{innercustomgeneric}{\customgenericname}
\providecommand{\customgenericname}{}
\newcommand{\newcustomtheorem}[2]{%
  \newenvironment{#1}[1]
  {%
   \renewcommand\customgenericname{#2}%
   \renewcommand\theinnercustomgeneric{##1}%
   \innercustomgeneric
  }
  {\endinnercustomgeneric}
}

\newcustomtheorem{customthm}{Theorem}
\newcustomtheorem{customlemma}{Lemma}

\theoremstyle{remark}

\newtheorem*{remark}{Remark}

\newtheorem{example}{Example}

\usepackage{fullpage}

 \newcommand{\CC}{\mathbb{C}}
\newcommand{\DD}{\mathbb{D}}
\newcommand{\TT}{\mathbb{T}}

\newcommand{\PP}{\mathbb{P}}
\newcommand{\RR}{\mathbb{R}}

\def\bb{\begin{color}{blue}}
\def\bg{\begin{color}{green}}
\def\br{\begin{color}{red}}
\def\eg{\end{color}}
\def\er{\end{color}}
\def\eb{\end{color}}

\let \le \leqslant
\let \leq \leqslant

\let \ge \geqslant
\let \geq \geqslant

\newcommand{\ee}{\epsilon}

\newcommand{\dx}{\textrm{d}}

\newcommand{\RRe}{\text{Re} \,}

\newcommand{\dtt}{\, \dtt}

\newcommand{\capc}{\mathbf{c}}

\newcommand{\parsig}{{\boldsymbol \sigma}}

\renewcommand{\epsilon}{\varepsilon}

\begin{document}

\bibliographystyle{alpha}

\title{One-dimensional scaling limits in a planar Laplacian random growth model}

\author{Alan Sola
\thanks{sola@math.su.se}}
\affil{Department of Mathematics, Stockholm University, 106 91 Stockholm, Sweden.}
\author{Amanda Turner
\thanks{a.g.turner@lancaster.ac.uk}}
\affil{Department of Mathematics and Statistics, Lancaster University, 
Lancaster LA1 4YF, UK.}
\author{Fredrik Viklund
\thanks{fredrik.viklund@math.kth.se}}
\affil{Department of Mathematics, Royal Institute of Technology, 100 44 Stockholm, Sweden.}
\maketitle

\begin{abstract}
We consider a family of growth models defined using conformal maps in which the local growth rate is determined by $|\Phi_n'|^{-\eta}$, where $\Phi_n$ is the aggregate map for $n$ particles. We establish a scaling limit result in which strong feedback in the growth rule leads to one-dimensional limits in the form of  straight slits. More precisely, we exhibit a phase transition in the ancestral structure of the growing clusters: for $\eta>1$, aggregating particles attach to their immediate predecessors with high probability, while for $\eta<1$ almost surely this does not happen.
\end{abstract}

\tableofcontents

\section{Introduction}\label{intro}
\subsection{Conformal aggregation processes}
Laplacian growth models describe processes where the local growth rate of a piece of the boundary of a growing compact cluster is determined by the Green's function of the exterior of the cluster. Such growth processes can be used to model a range of physical phenomena, including ones involving aggregates of diffusing particles. Discrete versions can be formulated on a lattice in all dimensions: some famous examples of this type of growth process include diffusion-limited aggregation (DLA) \cite{WS81}, the Eden model \cite{E61}, or the more general dielectric breakdown model (DBM) \cite{NPW84}. Despite considerable numerical evidence suggesting that the clusters that arise in these processes exhibit fractal features, very few rigorous results are known (for DLA, see \cite{Kes87}) and it remains a formidable challenge to rigorously analyze long-term behavior such as sharp growth rates of the clusters.

One objection that can be leveled at lattice-based models is that the underlying discrete spatial structure could potentially introduce anisotropies in the growing clusters that are not present in the physical setting of the plane or three-dimensional space. Indeed, large-scale simulations in two dimensions demonstrate anisotropy along the coordinate axes \cite{GB17}.  This fact provides one motivation for the study of off-lattice versions of aggregation processes. In the plane, such off-lattice models can be formulated in terms of iterated conformal mappings,  providing access to complex analytic machinery. Clusters produced by these conformal growth processes are initially isotropic by construction, but simulations suggest that in many instances, anisotropic structures appear on timescales where the number of aggregated particles becomes large compared to the size of the individual constituent particles. Nevertheless, proving the existence of such small-particle limits, whether anisotropic or not, has proved elusive, similarly to the case of lattice-based models. 

A fascinating feature of Laplacian growth models is competition between concentration and dispersion of particle arrivals on the cluster boundary. Protruding structures (``branches") and their endpoints (``tips") tend to attract relatively many arrivals, but they compete with each other as well as the remainder of the boundary. (Kesten's discrete Beurling estimate gives an upper bound on the tip concentration in the case of DLA.) The degree to which tips are favored is determined by the exact choice of growth rule, and several models contain one or more parameters that affect concentration, dispersion, and competition \cite{NPW84, HL98, CM02, Law06}. 

Previous and recent work on small-particle limits of conformal aggregation models \cite{NT12, JST15, Silv, NST19} has yielded growing disks, that is, smooth and isotropic shapes; the dispersion effect ``wins'' in the limit. In this paper, we study a particular instance of a conformal growth model, focusing instead on the concentration aspect of Laplacian growth and showing that anisotropic scaling limits arise in the presence of strong feedback in the growth rule. The scaling limits we exhibit are highly degenerate in the sense that growth, which is initially spread out, favors tips very strongly, and eventually collapses onto a single growing slit. 

To state our results, we first describe the general class of processes our object of study fits into. Let $\capc>0$, and let $f_{\capc}$ denote the unique conformal map
\[f_{\capc}\colon \Delta=\{z\in \CC\colon |z|>1\}\cup\{\infty\} \to D_1=\Delta\setminus (1,1+d]\]
having $f_{\capc}(z)=e^{\capc}z+\mathcal{O}(1)$ at infinity, and sending the exterior disk $\Delta$ to the complement of the closed unit disk with a slit of length $d=d(\capc)$ attached 
to the unit circle $\mathbb{T}$ at the point $1$. 
The logarithmic capacity $\capc$ and the length $d$ of the slit satisfy
\begin{equation}
e^{\capc}=1+\frac{d^2}{4(1+d)};
\label{capcvssize}
\end{equation}
in particular, $d\asymp \capc^{1/2}$ as $\capc\to 0$.
In terms of aggregation, the closed unit disk can be viewed as a seed, while the slit represents an attached particle. Typically, we think of the particle as being small compared to the seed.

A general two-parameter framework to model random or deterministic aggregation, based on conformal maps, is given by the following construction. Pick a sequence $\{\theta_k\}_{k=1}^{\infty}$ in
$[-\pi, \pi)$, and let $\{d_k\}_{k=1}^{\infty}$, or, equivalently, $\{c_k\}_{k=1}^{\infty}$, be a sequence of non-negative numbers connected via \eqref{capcvssize}. From the two numerical sequences $\{\theta_k\}$ and $\{c_k\}$, we obtain a sequence $\{f_k\}_{k=1}^{\infty}$ of rotated and rescaled conformal maps, referred to as building blocks, via
\[f_k(z)=e^{i\theta_k}f_{c_k}(e^{-i\theta_k}z).\]
On its own, each individual $f_k$ grows a slit in the exterior disk, attached at $e^{i\theta_k}$ and having logarithmic capacity $c_k$.
Finally, we set
\begin{equation}
\Phi_n(z)=f_1\circ\cdots\circ f_n(z), \quad n=1,2,\ldots.
\label{aggproc}
\end{equation}
Each $\Phi_n$ is itself a conformal map sending the exterior disk onto the complement of a compact set $K_n \subset \CC$, that is,
\[\Phi_n\colon \Delta \to \CC\setminus K_n.\]
The sets $\{K_n\}_{n=1}^{\infty}$ are called clusters. They satisfy $K_{n-1}\subset K_n$, and model a growing two-dimensional aggregate formed of $n$ particles. At infinity, we have 
\[\Phi_n(z)=e^{C_n}z+\mathcal{O}(1),\]
where 
\begin{equation}
\mathrm{cap}(K_n)=e^{C_n}=e^{\sum_{k=1}^{n}c_k}
\label{totalcaprel}
\end{equation}
is the total capacity of the $n$th cluster.

When modeling random aggregates formed via diffusion, one  chooses the angles $\{\theta_k\}$ to be i.i.d., and uniform in $[-\pi,\pi)$. Due to the conformal invariance of harmonic measure, this has the effect of attaching the $n$th particle at a point chosen according to harmonic measure (seen from infinity) on the boundary of $K_{n-1}$. This type of setup has been considered in a number of papers, see for instance \cite{HL98, Mak99, CM01, MatJen02, RZ05, JS09, NT12, JST12, JST15, Silv}; we shall only briefly mention models that are particularly pertinent to our study.

\subsection{Aggregate Loewner Evolution (ALE)}
The main object of study in the present paper is a model we refer to as aggregate Loewner evolution, abbreviated $\mathrm{ALE}(\alpha, \eta)$, with parameters $\alpha \in \RR$ and $\eta \in \RR$. In $\mathrm{ALE}(\alpha, \eta)$, conformal maps $\Phi_n$ are defined as in \eqref{aggproc} as follows.

Initialize by setting $\Phi_0(z)=z$ and letting $\mathcal{F}_{0}$ be the trivial $\sigma$-algebra.
\begin{itemize}
\item{ 
For $k=1, 2, 3, \ldots$, we let $\theta_k$ have distribution conditional on $\mathcal{F}_{k-1}=\mathcal{F}(\theta_1,\ldots \theta_{k-1}; c_1, \ldots, c_{k-1})$ given by
\begin{equation}
h_k(\theta)=\frac{|\Phi_{k-1}'(e^{\parsig+i\theta})|^{-\eta}\dx \theta}{\int_{\TT}|\Phi_{k-1}'(e^{\parsig+i\theta})|^{-\eta}\dx \theta}.
\label{aleangles}
\end{equation}
Here, $\parsig>0$ is a regularization parameter, which ensures that the angle distributions are well defined even though $\Phi'_{k-1}(e^{i\theta})$ has zeros and singularities on $\TT$. The parameter $\parsig$ is allowed to depend on the basic logarithmic capacity parameter $\capc$. Typically, we shall take \[\parsig=\parsig(\capc)=\capc^{\gamma}\] for some appropriate $\gamma>0$.}
\item{
Next, we define a sequence of logarithmic capacities for $k=1, 2,3, \ldots$ by taking
\begin{equation}
c_k=\capc |\Phi_{k-1}'(e^{\parsig+i\theta_k})|^{-\alpha}.
\label{alecaps}
\end{equation}
}
\end{itemize}
We note that $\mathrm{ALE}(\alpha,0)$ is the same model as the Hastings-Levitov $\mathrm{HL}(\alpha)$ model studied in \cite{HL98, Davetal99, RZ05, JST15}, and in particular 
 $\mathrm{ALE}(0,0)$ coincides with the $\mathrm{HL}(0)$ model studied in depth in \cite{NT12, Silv}. The Hastings-Levitov model was introduced as a conformal mapping model of dielectric breakdown (DBM) \cite{NPW84}, a discrete model in which vertices are added to a growing cluster by drawing bonds from among the neighboring lattice points. At stage $n$ of $\mathrm{DBM}(\eta)$, a point is added to the cluster $K_n$ by including a neighbor of $(j,k) \in K_n$ with probability
\[p_n\left((j,k)\rightarrow (j',k')\right)=\frac{\phi_n(j',k')^{\eta} }{\sum_{(l,m)} \phi_n(l,m)^{\eta}}.\] 
Here, summation is over lattice neighbors of $K_n$ and the function $\phi_n$ is discrete harmonic, and has $\phi_n=0$ on $K_n$ and $\phi_n=1$ on some large external circle. 

Off-lattice versions of $\mathrm{DBM}$ involving non-uniform angle choices determined by the derivative of a conformal map have been considered by several authors. Hastings \cite{Has01}, and subsequently Mathiesen and Jensen \cite{MatJen02}, study a model that essentially corresponds to $\mathrm{ALE}(2,\eta)$ modulo a slightly different parametrization in $\eta$. (In fact, an alternative name for the growth model in this paper could have been $\mathrm{DBM}(\alpha, \eta)$ or $\mathrm{HL}(\alpha, \eta)$, but we have opted for a different terminology to avoid confusion with lattice models, and also to emphasize connections with the Loewner equation, see below.) Hastings argues that for large enough exponents, more precisely, for $\eta\geq 3$ in our parametrization, the corresponding clusters become one-dimensional; he also points out that the behavior of the models depends strongly on the choice of regularization.

Another model that fits into this general framework is the Quantum Loewner Evolution model ($\mathrm{QLE}(\gamma, \eta)$) of Miller and Sheffield \cite{MSDuke,MSLQG} which is proposed as a scaling limit of DBM($\eta$) on a $\gamma$-Liouville quantum gravity surface. In the $\mathrm{QLE}$ construction, particles are attached according to a distribution which depends on the power of the derivative of the cluster map, as in \eqref{aleangles}, but with an additional term involving the Gaussian Free Field due to the presence of Liouville quantum gravity. In the construction of $\mathrm{QLE}$, capacity increments are kept constant, as for $\mathrm{ALE}(0,\eta)$. However, each particle in $\mathrm{QLE}$ is constructed as an $\mathrm{SLE}$ curve, rather than the straight slits used in $\mathrm{ALE}$.

Common to all conformal mapping models of Laplacian growth is the difficulty that derivatives of conformal mappings do not remain bounded away from 0 or $\infty$ as they approach the boundary and therefore the map $\theta \mapsto |\Phi'_n(e^{i \theta})|^{-1}$ can be very badly behaved. For instance, even when $n=1$, $|\Phi'_n(e^{i \theta})|^{- \eta}$ is not integrable over $\mathbb{T}$ for certain values of $\eta$ and hence the $\mathrm{ALE}(\alpha,\eta)$ model would not be well defined if we were to use $|\Phi'_n(e^{i \theta})|^{-\eta}$ as angle density. As mentioned above, for this reason we define the model via the regularization parameter $\parsig$ as in \eqref{aleangles}, and then let $\parsig \to 0$ together with the (pre-image) particle size, controlled by the parameter $\capc$. A similar difficulty arises from the dependence of the particle sizes on the derivatives of the conformal mappings. Although in this case the model is well-defined without the need for a regularization parameter in \eqref{alecaps}, it is no longer possible to guarantee that the resulting clusters have total capacity bounded above and below. Indeed, even with the presence of a regularization parameter, it is not clear that the total capacity remains bounded as $\parsig \to 0$. The exception is the $\mathrm{ALE}(0,\eta)$ model: in light of \eqref{totalcaprel}, taking $n\asymp \capc^{-1}$ is a natural choice of time-scaling in $\mathrm{ALE}(0,\eta)$ as with this choice the resulting clusters have total capacity bounded above and below. This in turn means that the total diameter of the clusters $K_n$ remains bounded as a consequence of Koebe's $1/4$-theorem, see \cite{PomBook}. The fact that we have some a priori control over the global size of clusters is our main motivation for moving from studying $\mathrm{HL}(\alpha)$ with $\alpha$ large to $\mathrm{ALE}(0,\eta)$ with $\eta$ large. Simulations suggest that one-dimensional limits are present also in $\mathrm{HL}(\alpha)$ for large $\alpha$ but showing that this is the case seems technically more difficult.

In this paper, we mainly focus on $\mathrm{ALE}(0,\eta)$ for $\eta>1$, and show that the conformal maps $\Phi_n$ converge to a randomly oriented single-slit map in the regime where 
$n\asymp \capc^{-1}$. 
This can be viewed as a rigorous version of Hastings' investigation \cite{Has01} of $\mathrm{ALE}(2,\eta)$ for the $\mathrm{ALE}(0,\eta)$ model. To obtain our convergence results, we exploit what is in a way the most extreme mechanism that could lead to a single-slit limit, namely that of aggregated particles becoming attached to their immediate predecessors. The main difficulties in the proof are that the angle densities induced by slit maps exhibit bad behavior even in the presence of regularization and have maxima and minima of different orders in the regularization parameter $\parsig$, making it hard to show convergence to a point mass. Furthermore, the feedback mechanism in \eqref{aleangles} is sensitive so that a single ``bad" angle can destroy the genealogical structure of the growing slit by leading to the creation of a new, competing tip further down the slit, which could lead to a splitting of growth into two branches.

\section{Overview of results}\label{overview}
\begin{figure}[ht!]
    \subfigure[$\eta=-1.0$]
      {\includegraphics[width=0.4 \textwidth]{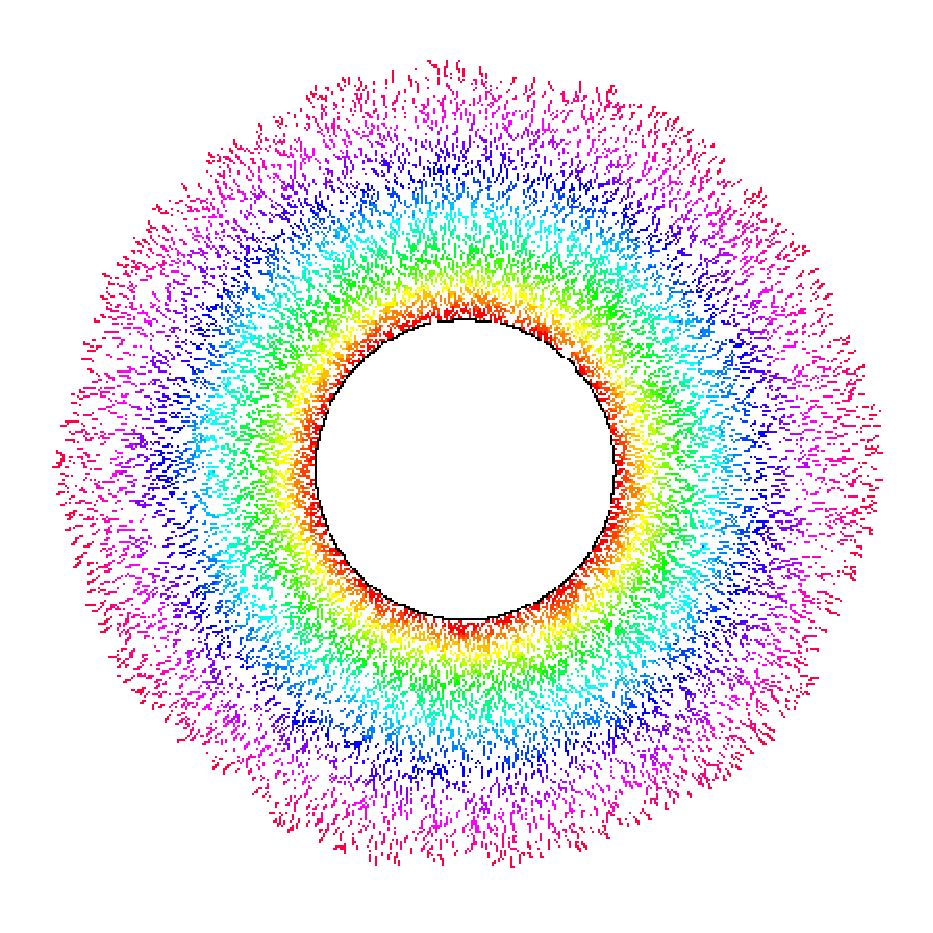}}
    \hfill
    \subfigure[$\eta=0.0$]
      {\includegraphics[width=0.4 \textwidth]{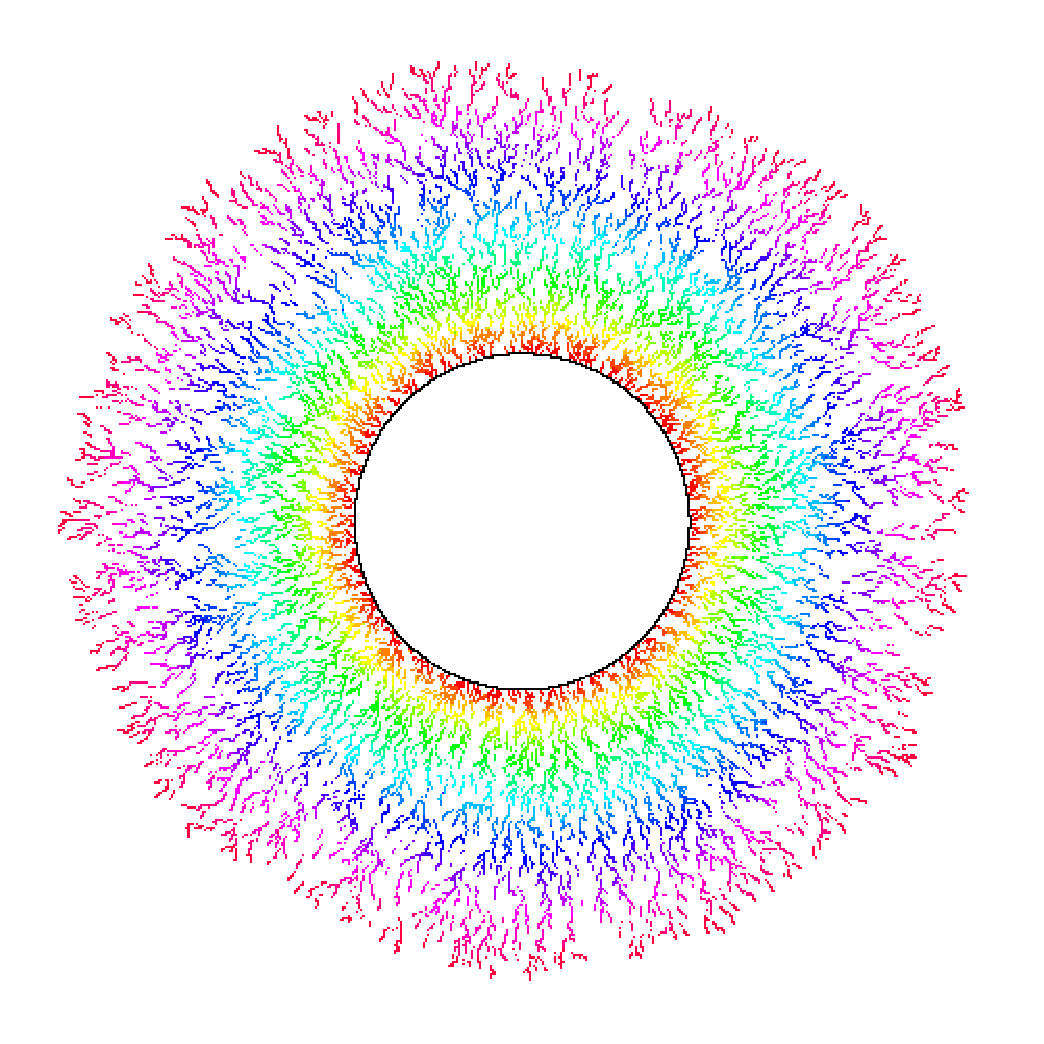}}
		\\
    \subfigure[$\eta=1.0$]
      {\includegraphics[width=0.4 \textwidth]{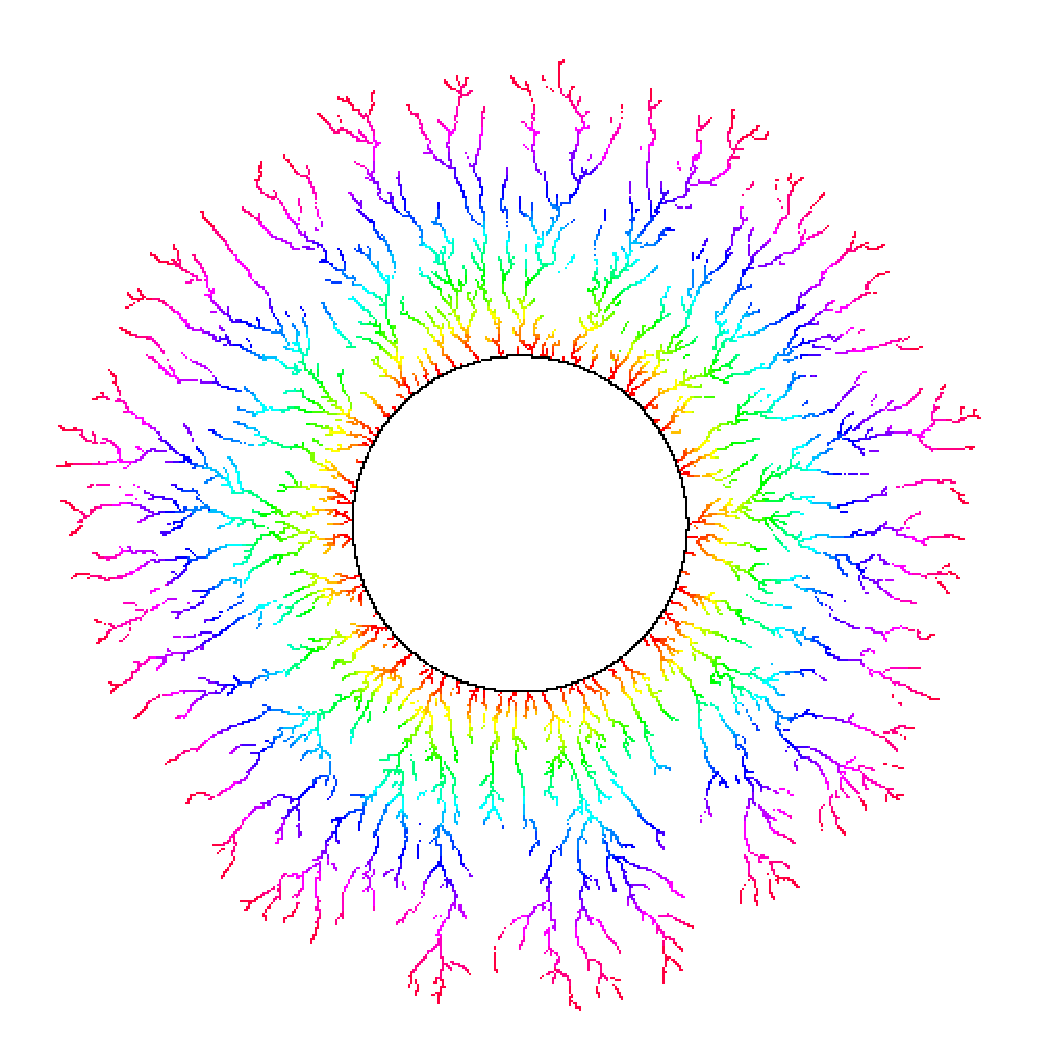}}
    \hfill
    \subfigure[$\eta=1.5$]
      {\includegraphics[width=0.4 \textwidth]{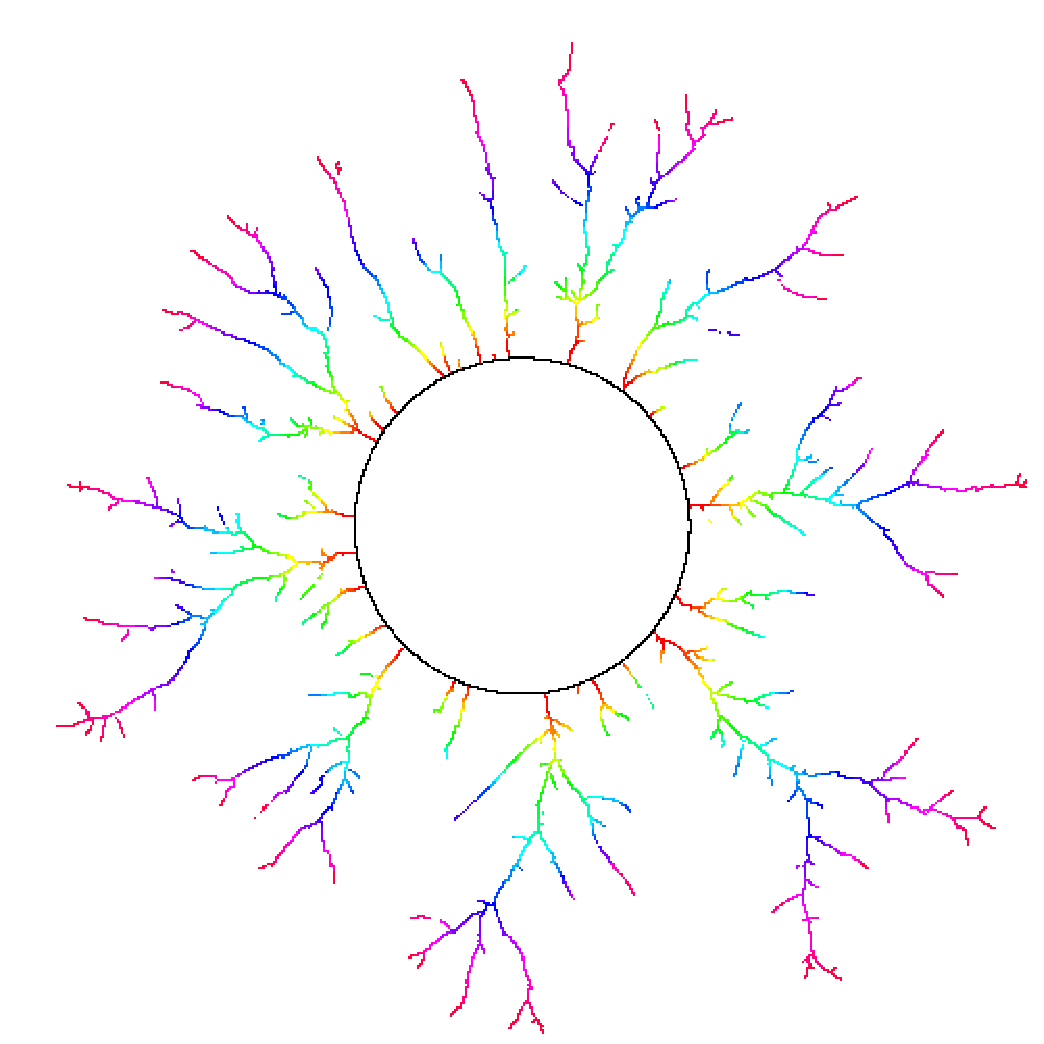}}
    \\
		\subfigure[$\eta=2.0$]
      {\includegraphics[width=0.45 \textwidth]{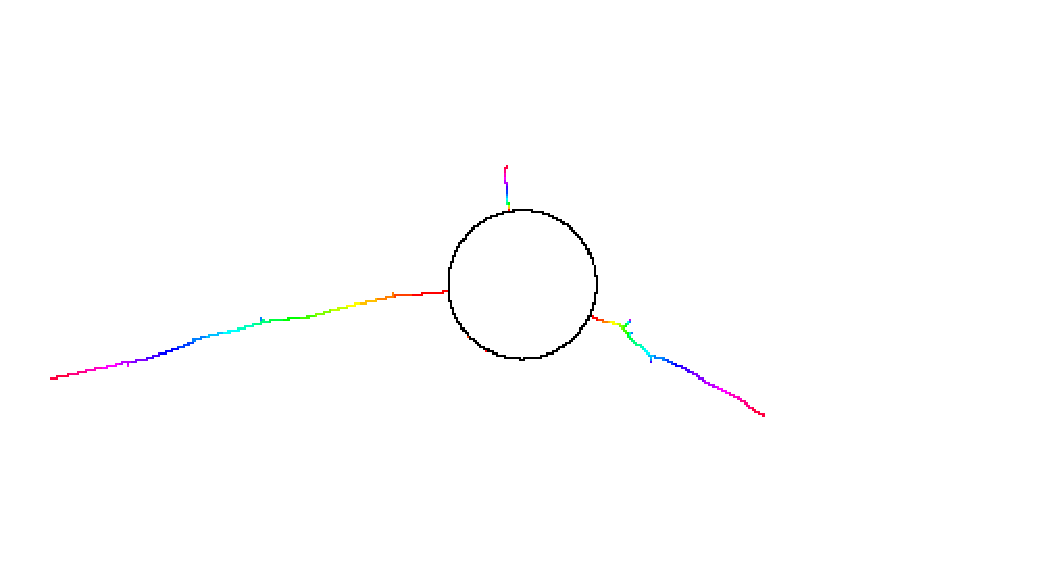}}
    \hfill
    \subfigure[$\eta=4.0$]
      {\includegraphics[width=0.45 \textwidth]{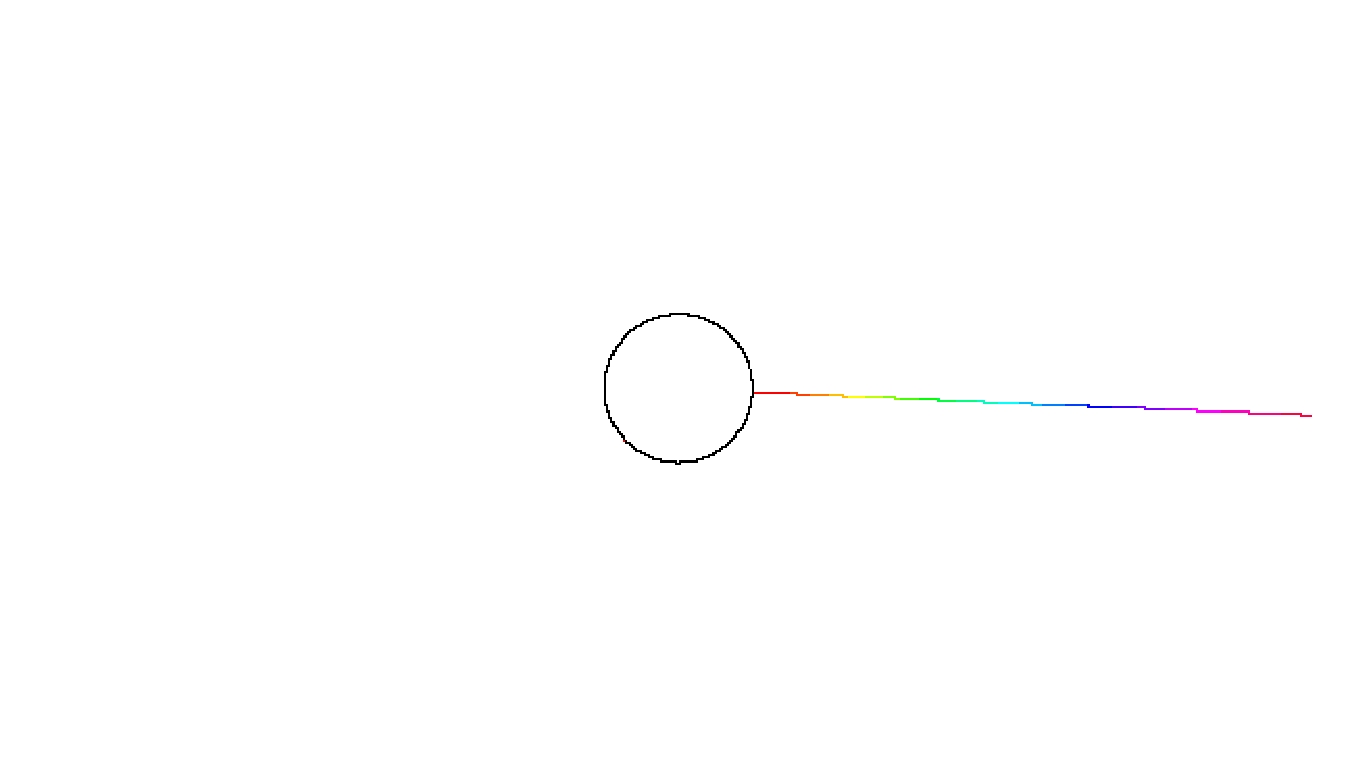}}
  \caption{\textsl{$\mathrm{ALE}(0, \eta)$ clusters with $\capc=10^{-4}$, $\parsig=\capc^2$, and $n=10,000$.}}
\label{alesimulations}
\end{figure}
Clusters that are formed by successively composing slit maps come with a natural notion of ancestry for their constituent particles. We say that a particle $j$ has parent 0 if it attaches directly to the unit disk and that the particle $j$ has parent $k$ if the $j$th particle is directly attached to the $k$th particle for $j>k$. More precisely, suppose that $\beta_{\capc} \in (0, \pi)$ is defined by
\[
f_{\capc}^{-1}\left ( (1, 1+ d(\capc)] \right ) = \{ e^{i \theta} : |\theta|<\beta_{\capc}\}
\]
so $e^{\pm i\beta_{\capc}}$ is mapped by the basic slit map to the base point of the slit i.e.~$f_{\capc}(e^{ \pm i \beta_{\capc}}) = 1$. Therefore particle $j$ has parent $0$ if $|\Phi_{j}(e^{i (\theta_j \pm \beta_{\capc})})|=1$ and particle $j$ has parent $k\geq 1$ if
\[
e^{-i \theta_{k}} \Phi_{k,j}(e^{i (\theta_j \pm \beta_{\capc})}) \in  (1, 1+ d(\capc)],
\]
where $\Phi_{k,j}(z)=f_k \circ f_{k+1} \circ \cdots \circ f_j(z)$.

In the $\mathrm{ALE}(0,\eta)$ model, each successive particle chooses its attachment point on the cluster according to the relative density of harmonic measure (as seen from infinity) raised to the power $\eta$. As the highest concentration of harmonic measure occurs at the tips of slits, intuitively one would expect that for sufficiently large values of $\eta$ each particle is likely to attach near the tip of the previous particle. In this paper we show that this indeed happens, and we identify the values of $\eta$ for which the above event occurs with high probability in the small-particle limit, that is, we show that the probability tends to $1$ as $\capc \to 0$. Figure \ref{alesimulations} displays $\mathrm{ALE}(0,\eta)$ clusters for different values of $\eta$.

\begin{figure}[ht!]
\begin{center}
    \subfigure[$\parsig=\capc^{1/4}$]
      {\includegraphics[width=0.35 \textwidth]{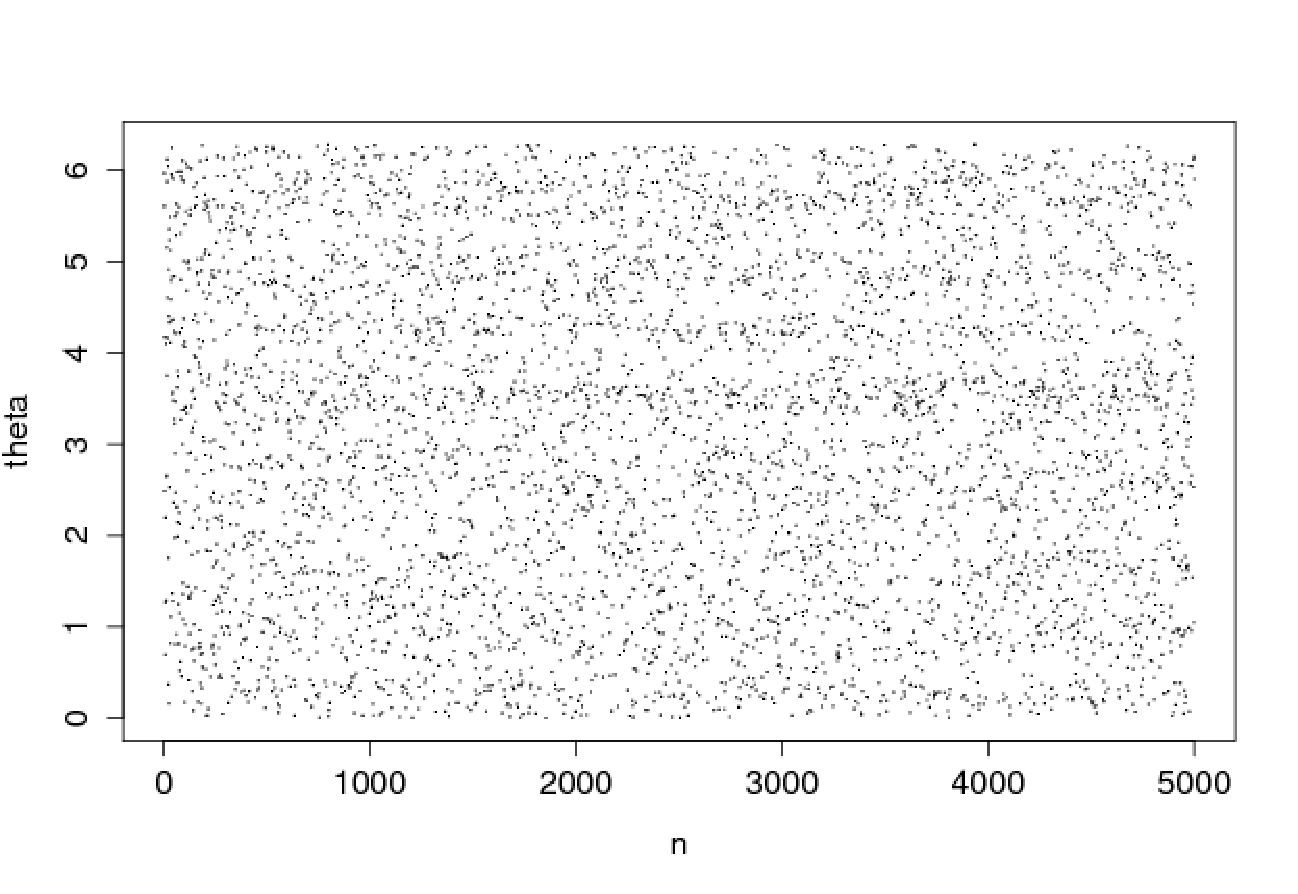}}
    \hfill
    \subfigure[$\parsig=\capc^{1/2}$]
      {\includegraphics[width=0.35 \textwidth]{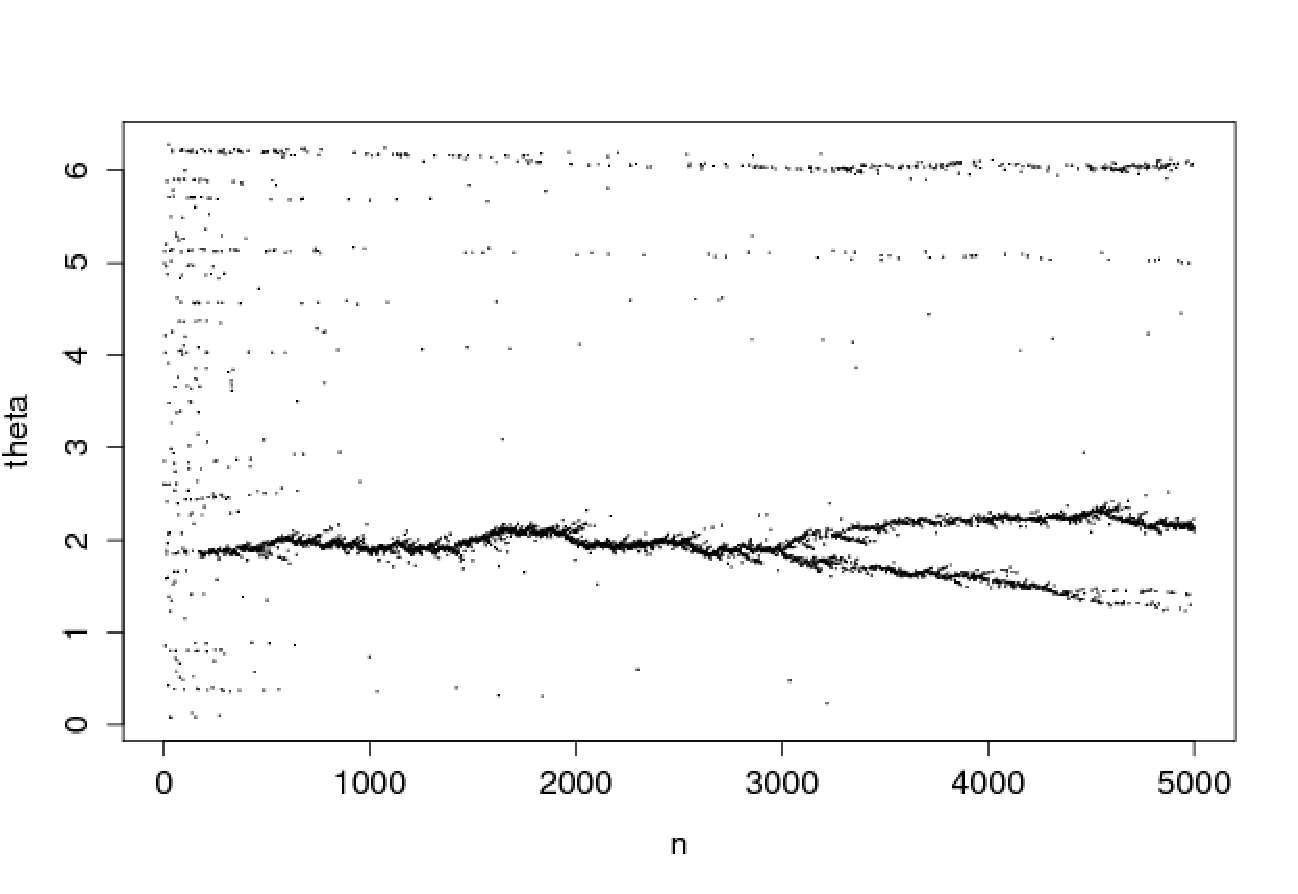}}
    \subfigure[$\parsig=\capc$ ]
      {\includegraphics[width=0.35 \textwidth]{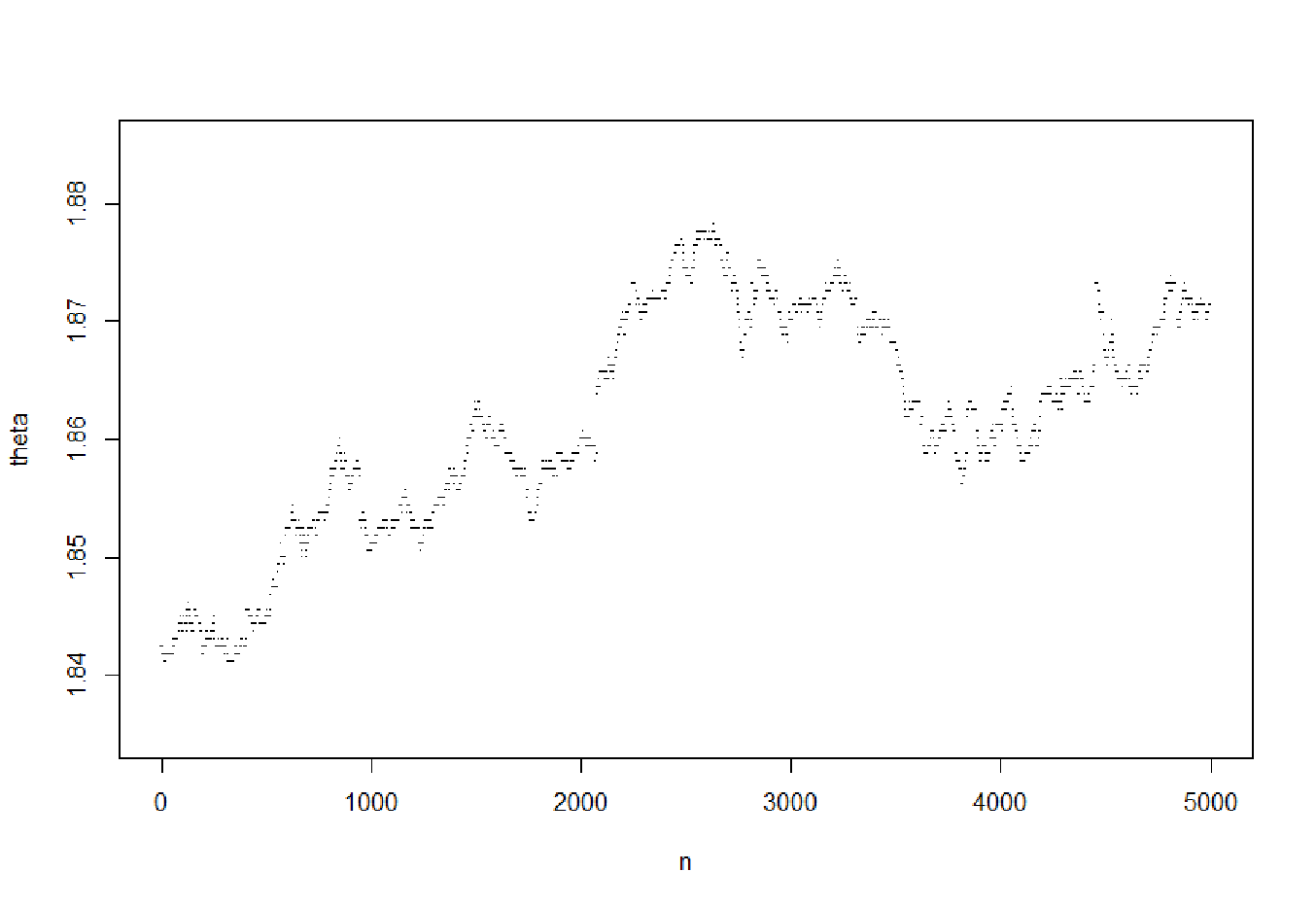}}
    \hfill
    \subfigure[$\parsig=\capc^{2}$]
      {\includegraphics[width=0.35 \textwidth]{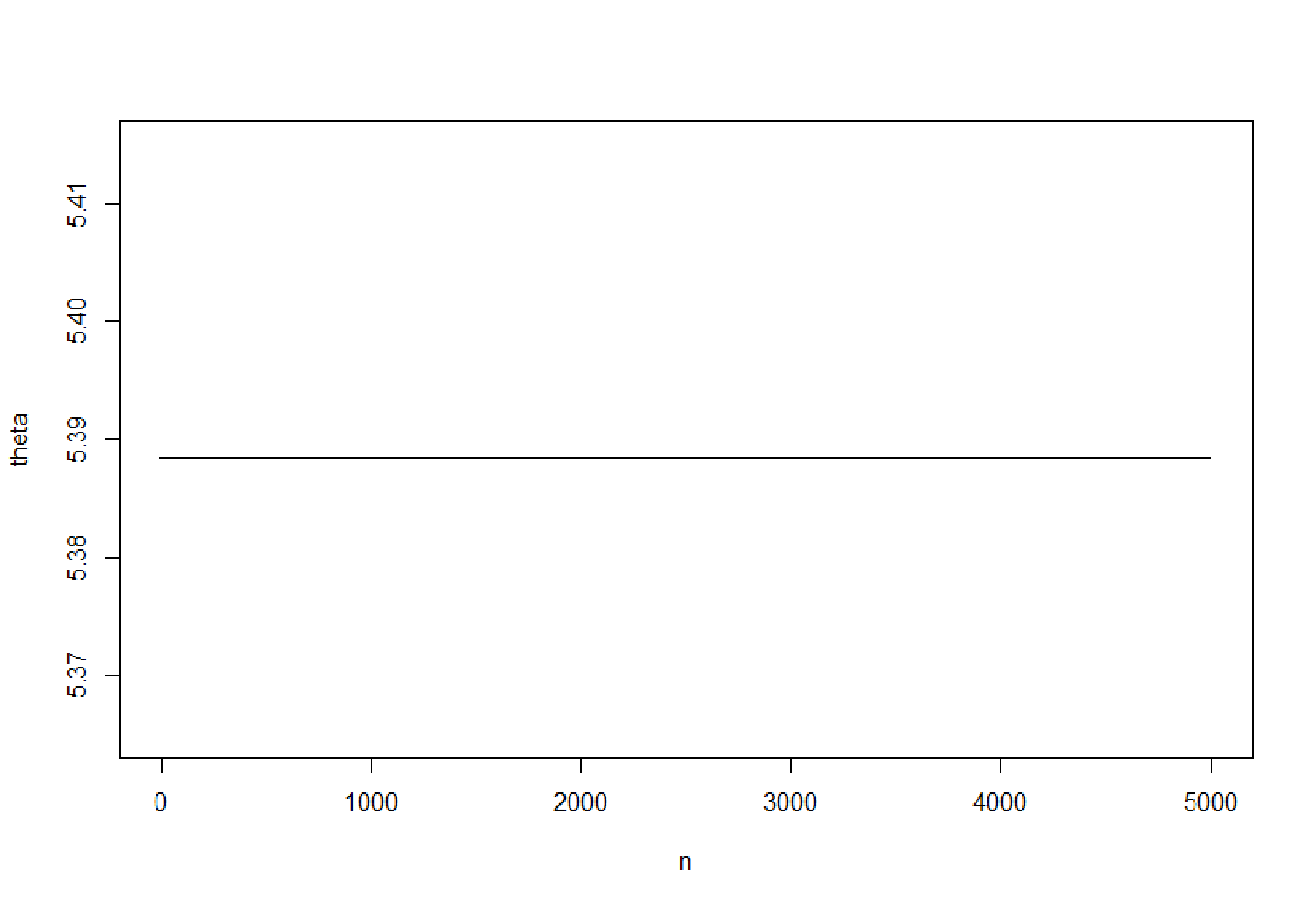}}
  \caption{\textsl{$\mathrm{ALE}(0,4)$ angle sequences with $\capc=10^{-4}$ and $n=5,000$, with varying regularization $\parsig$. (Note that images (c) and (d) are on a different spatial scale to (a) and (b), but the same spatial scale as each other.)}}
\label{sigmavaries}
\end{center}
\end{figure}

The limiting behavior of the model is quite sensitive to the rate at which $\parsig \to 0$ as $\capc \to 0$. Figure \ref{sigmavaries} shows how the angle sequences $\{\theta_k\}$ in $\mathrm{ALE}(0,4)$ are affected by the choice of exponent $\gamma$ when regularizing by $\parsig=\capc^{\gamma}$. This phenomenon is also observed by Hastings in \cite{Has01} for a related model. In \cite{JST15}, which deals with slow-decaying $\parsig$ scaling limits in a strongly regularized version of $\mathrm{HL}(\alpha)$, it is shown that the scaling limits of the clusters are disks for all values of $\alpha \geq 0$, provided $\parsig \gg (\log \capc^{-1})^{-1/2}$. By using similar techniques, combined with those developed in the paper \cite{NST19}, it is possible to prove that the corresponding scaling limits in $\mathrm{ALE}(0,\eta)$ are again disks for all $\eta \in \mathbb{R}$, provided $\parsig \gg (\log \capc^{-1})^{-1}$. (In \cite{NST19}, which focusses on the case $\eta \leq 1$, the stronger result is obtained that $\mathrm{ALE}(0,\eta)$ clusters converge to disks for all $\parsig \gg \capc^{\gamma}$ where $\gamma = 1/3$ if $\eta<1$ or 1/5 if $\eta=1$, and a phase-transition is observed at $\eta=1$ at the level of fluctuations). Together with the result in Theorem \ref{thm:main-thm} stated below, this shows the existence of a transition in the macroscopic shape of the $\mathrm{ALE}(0,\eta)$ clusters when $\eta > 1$, from slits to disks as the regularization parameter $\parsig$ increases. Simulations  suggest that there might be an intermediate regime where a suitable spatial rescaling, as in Figure \ref{sigmavaries}(c), reveals stochastic features in the angle sequence $\{\theta_n\}$. 
As we seek results in this paper which do not strongly depend on the choice of regularisation parameter, part of our objective is to identify the minimal value of $\eta$  for which there exists some $\sigma_0$ (dependent on $\capc$ and $\eta$) such that, provided $\parsig<\sigma_0$, with high probability each particle lands on the tip of the previous particle.

The following is the main result of the paper and shows that the $\mathrm{ALE}(0,\eta)$ model exhibits a phase transition at $\eta=1$ in the genealogy of the growing cluster in the small-particle limit. See Theorem \ref{slitstheorem} for a complete statement and proof; in particular we give sufficient conditions on $\gamma$.

\begin{theorem}[$\mathrm{ALE}(0,\eta)$ model]\label{thm:main-thm}
For $\mathrm{ALE}(0,\eta)$ with logarithmic capacity parameter $\capc$ and regularization parameter $\parsig$, let $\Omega_{N}=\Omega_{N}^{\eta,\capc, \parsig}$ be the event defined by
\[
\Omega_N = \{ \mbox{Particle $j$ has parent $j-1$ for all } j=1, \dots, N \}.
\]
For each $\eta>1$, there exists some $\gamma=\gamma(\eta)$ such that if $\sigma_0=\capc^\gamma$ and if $N=n(T):=\lfloor T \capc^{-1} \rfloor$ for some fixed $T>0$, then
\[
\lim_{\capc \to 0} \inf_{0<\parsig<\sigma_0} \mathbb{P}(\Omega_N) = 1,
\]
whereas if $\eta<1$, then for any $N >1$,
\[
\lim_{\capc \to 0} \sup_{\parsig>0} \mathbb{P}(\Omega_N) =0.
\]

In the case when $\eta>1$ and $\parsig<\sigma_0$, it follows that, for any $r>1$ and $T<\infty$,
\[
\sup_{t \leq T} \sup_{\{|z|>r\}}|\Phi_{n(t)}(z) -e^{i\theta_1} f_t(e^{-i\theta_1}z) | \to 0 \quad \textrm{ in probability as }\quad \capc\to 0,
\]
and the cluster $K_{n(t)}$ converges in the Hausdorff topology to a disk with slit of logarithmic capacity $t$ attached at position $z=e^{i\theta_1}$.
\end{theorem}

\subsection{A related Markovian model}

Observe that, for each $k$, we are free to specify the interval of length $2 \pi$ in which to sample $\theta_k$, and this choice does not have any effect on the maps $\Phi_n$. It is convenient to choose to sample $\theta_k$ from the interval $[\theta_{k-1}-\pi, \theta_{k-1}+\pi)$. In this case, we can express the event as
\[
\Omega_N=\left\{ \sup_{2 \leq j \leq N} |\theta_{j} - \theta_{j-1}| < \beta_{\capc} \right\}.
\]
(Recall that, by definition, $\beta_{\capc} \in (0,\pi)$ and $e^{\pm i\beta_{\capc}}$ is mapped
 by the basic slit map to the base point of the slit i.e.~$f_{\capc}(e^{ \pm i \beta_{\capc}}) = 1$.) One of the main difficulties in analysing this event is that the distribution of $\theta_k$ conditional on $\mathcal{F}_{k-1}$ (as defined in \eqref{aleangles}), depends non-trivially on the entire sequence $\theta_1, \dots, \theta_{k-1}$. In this subsection, we introduce an auxiliary model for random growth in the exterior unit disk in which the sequence of attachment angles is Markovian. The Markov model is relatively straightforward to analyse and exhibits an analogous phase transition to that described above. The remainder of the paper is concerned with examining how ALE$(0,\eta)$ and the Markov model relate to each other. 

Set $\Phi_0^*(z)=z$ and let $\{\Phi^*_n\}$ be conformal maps obtained through composing
\[\Phi^*_n=f_1^*\circ \cdots \circ f_n^*,\]
where each $f_k^*$ is a building block with $c_k=\capc$, and rotation angle $\theta_k^*$ having conditional distribution with density
\begin{equation}
h_k^*(\theta|\theta^*_{k-1}) = \frac{1}{Z_{k-1}^*}|f'_{\capc (k-1)}(e^{\parsig+i(\theta-\theta^*_{k-1})})|^{-\eta}, \quad k=1,2,3,\ldots.
\label{auxangles}
\end{equation}
Here, we have set \[Z^*_{k}=\int_{\TT}|f'_{\capc k}(e^{\parsig+i \theta})|^{-\eta}\dx \theta\]
and suppressed the dependence on $\capc$, $\parsig$ and $\eta$ to ease notation. 

In order for the measure above to be well-defined when $\eta \geq 1$, we require $\parsig > 0$. In words, the density of the $k$th angle distribution in this model is obtained by replacing the complicated $(k-1)$th cluster map of ALE by a simple slit map ``centered'' at $\theta^*_{k-1}$, and with deterministic logarithmic capacity $\capc (k-1)$. 

For this model we obtain the following theorem: we again set $n(t)=\lfloor t/\capc\rfloor$, let $K_{n(t)}^*$ denote the cluster associated with $\Phi^*_{n(t)}$, and define the event
\[
\Omega_N^* = \{ \mbox{Particle $j$ in the $\ast$-model has parent $j-1$ for all } j=1, \dots, N \}.
\]

\begin{theorem}[Markov model]\label{thm: auxslitconv}
Set $\sigma_0 = \capc^{\gamma^*}$ where
\[
\gamma^* > \frac{\eta+1}{2(\eta-1)}.
\]
Then
\begin{align*}
\lim_{\capc \to 0} \inf_{0<\parsig<\sigma_0} \mathbb{P}(\Omega_N^*) &= 1 \quad \mbox{if } \eta > 1 \\
\lim_{\capc \to 0} \sup_{\parsig>0} \mathbb{P}(\Omega_N^*) &=0 \quad \mbox{if } \eta<1.  \label{clustersfig}
\end{align*} 
Furthermore, when $\eta > 1$ and $\parsig<\sigma_0$, for any $r>1$ and $T<\infty$,
\[
\sup_{t \leq T} \sup_{\{|z|>r\}}|\Phi^*_{n(t)}(z) - e^{i\theta_1^*}f_t(e^{-i\theta_1^*}z) | \to 0 \quad \textrm{ in probability as }\quad \capc\to 0,
\]
and the cluster $K^*_{n(t)}$ converges in the Hausdorff topology to a disk with slit of logarithmic capacity $t$ attached at position $z=e^{i\theta_1^*}$.
\end{theorem}
\begin{remark}
It can also be shown that $\lim_{\capc \to 0} \inf_{0<\parsig<\sigma_0} \mathbb{P}(\Omega_N^*) = 1$ when $\eta=1$, provided $\sigma_0 \to 0$ exponentially fast as $\capc \to 0$, but we omit the details here. 
\end{remark}
We give the relatively straight-forward proof of Theorem \ref{thm: auxslitconv} in Section \ref{Markovproof}. Because of the Markovian nature of the auxiliary model, all that is needed are estimates on the derivative of the explicit slit map to control the densities \eqref{auxangles}, together with standard martingale arguments.

\subsection{Overview of the proof of Theorem~\ref{thm:main-thm} and organization of the paper}
The main idea for the proof of Theorem \ref{thm:main-thm} is to show that the Markovian model of the previous section is a good approximation of the $\mathrm{ALE}(0,\eta)$ process. In order to do this one approach would be to try to argue that $|\Phi'_n(e^{\parsig+i\theta})|$ can be globally well approximated by $|(f^{\theta_{n}}_{n \capc})'(e^{\parsig + i \theta})|$, where we use the notation
\[
f_{\capc}^{\theta}(z) = e^{i\theta}f_{\capc}(e^{-i\theta}z)
\]
for the rotated slit maps. However, this seems difficult to make work to sufficient precision when evaluating the maps close to the boundary. Specifically, the map $\Phi'_n(z)$ has zeros (respectively singularities) at each of the points on the boundary of the unit disk which are mapped to the tip (respectively to the base) of one of the slits corresponding to an individual particle. In contrast, for the map $(f^{\theta_{n}}_{n \capc})'(z)$, the points corresponding to tips and bases of successive particles coincide and therefore the singularities and zeros corresponding to intermediate particles cancel each other out, leaving only a zero at the point mapped to the tip of the last particle and singularities at the two points which are mapped the base of the first particle (see Figure \ref{singularities}).

\begin{figure}[ht!]
\begin{center}
\includegraphics[width=8cm]{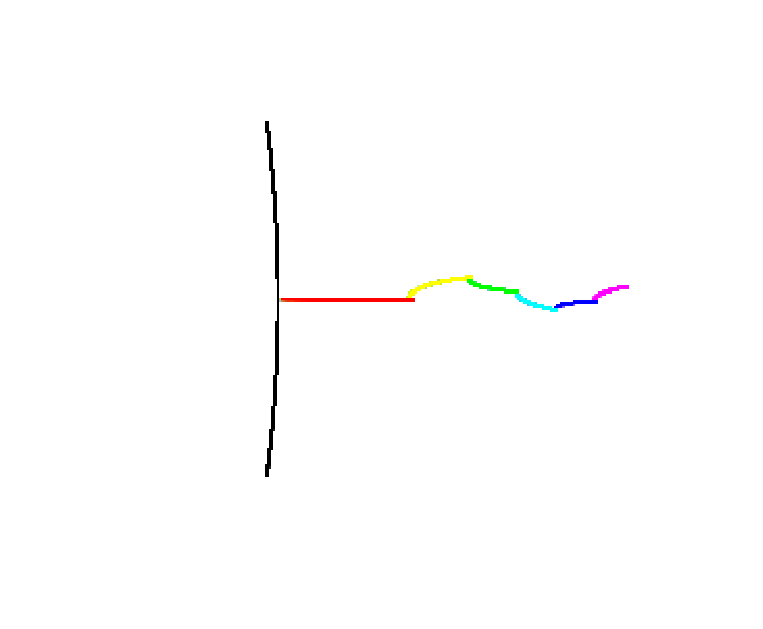}
\includegraphics[width=8cm]{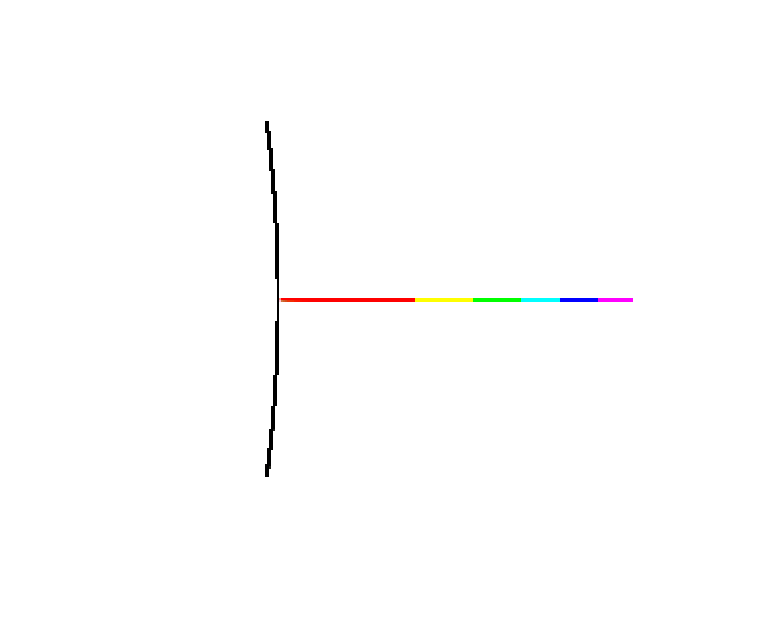}
  \caption{\textsl{Diagram illustrating the presence of zeros and singularities in the derivative at each successive particle tip and base in $\Phi_n(z)$ (left). These zeros and singularities are absent in $f_{n \capc}(z)$ except at the tip of the final particle and base of the first particle (right).}}
\label{singularities}
\end{center}

\end{figure}

Interactions between nearby tips can be subtle and are in general hard to analyze \cite{CM02}. Our strategy is instead to establish two properties of the distribution function $h_{n}(\theta)$. 
\begin{itemize}
\item{The first is to show that near the tip of the last particle to arrive the derivatives of $\Phi_n$ and $f_{n\capc}^{\theta_n}$ are in fact very close and so for very small values of $\theta-\theta_n$, $h_{n+1}(\theta)$ can be well approximated by $h_{n+1}^*(\theta|\theta_n)$.}
\item{ 
The second property is to show that $h_{n+1}(\theta)$ concentrates the measure so close to $\theta_{n}$ that even though the probability of attaching to earlier particles is higher than for the Markovian model, $\Omega_N$ still occurs with high probability, provided we now require
\[\gamma>
\begin{cases} 
(\eta^2+2\eta-1)/ [2(\eta-1)^2] &\mbox{ if } 1< \eta < 3; \\
(2\eta+1)/[2(\eta-1)] &\mbox{ if } 3 \leq \eta < 7; \\
5/4 &\mbox{ if } \eta \geq 7
\end{cases} \]
when regularizing by $\parsig<\sigma_0=\capc^{\gamma}$; see Figure \ref{glowerbounds} for plots of the lower bounds on $\gamma$ and $\gamma^*$.}
\end{itemize}

We now give a brief overview of the structure of the paper. In Section \ref{loewner} we provide some background information on the Loewner differential equation, which allows us to represent the aggregate maps $\Phi_n$ as solutions corresponding to a $[-\pi, \pi)$-valued driving process with equally spaced jump times and positions given by the random angles \eqref{aleangles}. In particular, we explain how convergence of an angle sequence $\{\theta_k\}$ allows us to deduce convergence of the corresponding conformal maps $\Phi_n$. 

In Section \ref{slitder} we obtain estimates on the derivative of the slit map used to construct the Markovian model. These estimates  lead to moment bounds for $[-\pi, \pi)$-valued random variables constructed from slit map derivatives. The arguments used are elementary in nature, and heavily use the explicit form of the slit map.

In Section \ref{slitconvergencesection}, we first apply our slit map estimates give a straight-forward proof of Theorem \ref{thm: auxslitconv}. 
Then we state the detailed estimates on Loewner derivatives at the tip and away from the approximate slit needed 
to show that $h_n(\theta)$, the density function for the $n$th angle $\theta_n$, has the required behaviour (deferring the proofs until the next Section). Similar arguments to those in the proof of Theorem \ref{thm: auxslitconv} are used to establish Theorem~\ref{thm:main-thm}, but since $\{\theta_k\}$ does not have a Markovian structure, there are further terms to control. We also discuss some extensions of our results, valid for certain instances of the $\mathrm{ALE}(\alpha,\eta)$ model as well as related models.

Finally, Section \ref{conformal} contains most of the technical machinery needed for the proof of Theorem \ref{thm:main-thm}. 
In this section, we obtain estimates on the distance between two solutions to the Loewner equation in terms of the distance between their respective driving functions,  in the case where we know what one of the solutions is (in our application it is a slit map). These estimates, which we believe may be of independent interest, enable us to obtain much more precise estimates than exist for generic solutions. In particular, our estimates give very good approximations when the conformal mappings are quite close to the boundary, whereas generic estimates blow up in this region. We perform this analysis by using the reverse-time Loewner flow \eqref{loewnerODE} to write the distance between the two solutions as the solution to an ordinary differential equation which we are able to linearize.

\begin{figure}
\begin{center}
\includegraphics[width=8cm]{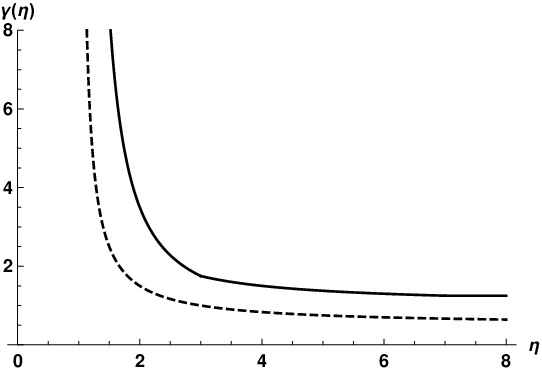}
 \caption{\textsl{Lower bounds on regularization exponents for $\mathrm{ALE}$ (solid) and the Markov model (dashed).}}
\label{glowerbounds}
\end{center}
\end{figure}

\subsubsection*{Notation}
Many of the estimates presented in this paper, especially in Section \ref{conformal}, are more precise than what is strictly needed for the proof of our main theorem, in that we frequently keep track of 
the dependence of constants on parameters, and similar. We have opted to record detailed versions to enable potential further applications where such dependencies may be important. Generic constants, which may change from line to line, will mainly be denoted by the capital letters $A$ and $B$.

Throughout, we use integer subscripts, or the letters $j$, $k$, and $n$, to denote building block maps, that is, rotated copies of a slit map aggregated to form the cluster maps $\Phi_n$ and $\Phi_n^*$. When we need to keep track of scaling, we use $f_{\capc k}$  (boldface subscript) to denote a slit map adding a single slit of logarithmic capacity $\capc k$ ($k=1,2,\ldots$) at the point $1$. Finally, a generic single-slit map centered at $1$ adding a slit of logarithmic capacity $t>0$ will be denoted $f_t$.

\section{Loewner flows}\label{loewner}
We shall make extensive use of Loewner techniques in this paper. Loewner equations describe the flow of families $\{\Psi_t\}_{t\geq 0}$ of conformal maps of a reference domain in $\CC\cup \{\infty\}$ onto evolving domains in the plane in terms of measures on the boundary.  We only give a very brief overview here, and refer the reader to \cite{Lawbook} and the references therein for a discussion of Loewner theory. 
\subsection{Loewner's equation}
Let $\{\mu_t\}_{t> 0}$ be a family of probability measures on the unit circle $\TT$, in this context referred to as driving measures, such that 
$t\mapsto \|\mu_t\|$ is locally integrable. Then the  Loewner partial differential equation for the exterior disk,
\begin{equation}
    \partial_t \Psi_t(z)=z\Psi_t'(z)\int_{\TT}\frac{z+\zeta}{z-\zeta} \dx \mu_t(\zeta) ,
\label{loewnerPDE}
\end{equation}
with initial condition 
\[\Psi_0(z)=z,\]
admits a unique solution $\{\Psi_t\}_{t\geq 0}$ called a Loewner chain \cite{Lawbook, CM01}. Each $\Psi_t(z)$ is a conformal map of the exterior disk onto a simply connected domain,
\[\Psi_t\colon \Delta\to D_t=\CC\cup\{\infty\}\setminus K_t\]
and at $\infty$ we have the power series expansion $\Psi_t(z)=e^{t}z+\mathcal{O}(1)$. The growing compact 
sets $\{K_t\}_{t\geq 0}$ are called hulls, satisfy $K_s\subsetneq K_t$ for $s<t$, and have $\mathrm{cap}(K_t)=e^t$ for $t\geq 0$, where 
$\textrm{cap}(K)$ denotes the  capacity of a compact set $K\subset \CC$.

The limit functions appearing in Theorem~\ref{thm:main-thm} can be realized in terms of Loewner chains, and in fact have a very simple Loewner representation.

\begin{example}[Growing a slit]
Let $\mu_{t}=\delta_1$, a point mass at $\zeta=1$. Then \eqref{loewnerPDE} reads
\[\partial_tf_t(z)=zf'_t(z)\frac{z+1}{z-1}.\]
With initial condition $f_0(z)=z$, the solution has the explicit representation (viz. \cite[p. 772]{MR05})
\begin{equation}
f_t(z)=\frac{e^{t}}{2z}\left( z^2+2(1-e^{-t})z+1+(z+1)\sqrt{z^2+2(1-2e^{-t})z+1}\right).
\label{slitexplicit}
\end{equation}
The solution precisely consists of the slit maps $f_t \colon \Delta\to \Delta \setminus (1,1+d(t)]$, where
\begin{equation}
d(t)=2e^t(1+\sqrt{1-e^{-t}})-2, \quad t>0.
\label{sticklength}
\end{equation}
This means that the growing hulls are $K_t=\overline{\DD}\cup(1,1+d(t)]$, the closed unit disk plus a radial slit emanating from $\zeta=1$. 
We note that the somewhat complicated expression in \eqref{slitexplicit} can be obtained by conjugating the simple formula for a slit map in the upper half-plane $\mathbb{H}=\{z\in  \CC \colon \mathrm{Im}(z)>0\}$, namely
\[
F_t(z)=\sqrt{z^2-4t},
\]
with suitable M\"obius transformations.
\end{example}
In this paper, we are mainly concerned with the case $\mu_t=\delta_{e^{i\xi_t}}$ for some function $\xi_t\colon (0,T] \to \RR$ and in that setting, we refer to $\xi_t$ as a driving term.

The conformal maps arising in $\mathrm{ALE}(\alpha, \eta)$ have the following simple Loewner representation. We first solve the Loewner equation with driving measure $\mu_t=\delta_{e^{i \xi_t}}$, where
\begin{equation}
\label{xidef}
\xi_t=\sum_{k=1}^N\theta_k\mathbf{1}_{(C_{k-1}, C_k]}(t),
\end{equation}
with $C_k=\sum_{j=1}^{k}c_k$, and the angles $\{\theta_k\}$ and logarithmic capacities $\{c_k\}$  given by \eqref{aleangles} and \eqref{alecaps}, respectively. Explicitly then, the Loewner problem associated with $\mathrm{ALE}(\alpha, \eta)$
reads
\begin{equation}
    \partial_t \Psi_t(z)=z\Psi_t'(z)\frac{z+e^{i\xi_t}}{z-e^{i\xi_t}} \quad \textrm{where}\quad \Psi_0(z)=z.
\label{loewnerPDEdrivingfunct}
\end{equation}
To obtain the composite $\mathrm{ALE}(\alpha, \eta)$-maps $\Phi_n$ described in Section \ref{intro}, we evaluate the solution to \eqref{loewnerPDEdrivingfunct} at $t=\capc n$; thus 
\[\Phi_n=\Psi_{\capc n}, \quad n=1,2,\ldots \,\,.\] The random driving function 
$\xi_t$ can be viewed as a c\`adl\`ag jump process exhibiting a complicated dependence structure encoded through angles and capacity increments. When $\alpha=0$, the dependence structure is only present in the distribution of the increments, as the jump times are deterministic, and equal to $\capc k$ for $k=1,2,\ldots$. We emphasize that this is the main technical reason why the $\mathrm{ALE}(0,\eta)$ model is easier to analyze then the general $\mathrm{ALE}(\alpha,\eta)$ model or the Hastings-Levitov model $\mathrm{HL}(\alpha)$.

\subsection{Reverse-time Loewner flow}
The Loewner equation \eqref{loewnerPDEdrivingfunct} is a first-order partial differential equation, and in the $\mathrm{ALE}(\alpha, \eta)$ model, it gives rise to a non-linear PDE problem since the driving measures depend on the maps $f_t$ via their derivatives. As is common in Loewner theory, we shall analyze solutions by passing to the backwards flow associated with \eqref{loewnerPDEdrivingfunct}: this essentially entails employing the method of characteristics to obtain an ordinary differential equation that describes the evolution at hand. See \cite{Lawbook, CM01} for detailed derivations and discussions.

Let $T>0$ be fixed. The equation for the backward or reverse-time flow in the exterior disk is
\begin{equation}
\partial_tu_t(z)=u_t(z)\frac{u_t(z)+e^{i\Xi_t}}{u_t(z)-e^{i\Xi_t}},
\label{loewnerODE}
\end{equation}
where we define
\[\Xi_t=\xi_{T-t}, \quad 0\leq t\leq T.\]
Then, setting $u_0(z)=z$, we obtain (see \cite[Chapter 4]{Lawbook})
\[u_T(z)=\Psi_T(z)\]
where $\Psi_t$ denotes the solution to the forward equation \eqref{loewnerPDEdrivingfunct} with driving function $\xi_t$. Note that this holds in general only at the special time $T$. 

The main advantage of the backward flow is the fact that, for each $z$, \eqref{loewnerODE} is now formally an ODE, simplifying the problem of analyzing and estimating the solution to the corresponding flow problem. Such analysis is carried out in Section \ref{conformal} and will be crucial in the proof of Theorem~\ref{thm:main-thm}.
 
\subsection{Convergence of Loewner chains}
Our strategy will be to argue that the driving function \eqref{xidef} arising in the $\mathrm{ALE}$ process is close, in the regime where $n\asymp \capc^{-1}$, to the constant driving function $\xi_t=\theta_1$. We would then like to argue that the resulting conformal maps are close. These kinds of continuity results have been established in several settings, see for instance \cite[Proposition 3.1]{JS09} and \cite[Proposition 1]{JST12}, and \cite{F1} for a more systematic discussion.

Since the $\mathrm{ALE}$ driving processes exhibit synchronous jumps, it is natural to measure distances between them in the uniform norm $\|\cdot\|_ {\infty}$. For $T>0$, we denote the space of piecewise continuous functions $\xi \colon [0, T) \to \RR$ endowed with this norm by $D_T$. We consider the space $\Sigma$ consisting of conformal maps $f(z)=Cz+\mathcal{O}(1)$ as $z\to \infty$, with $C>0$ uniformly bounded, and we endow $\Sigma$ with the topology of uniform convergence on compact subsets of $\Delta$. We then view the conformal maps $\Psi_t$, and hence the aggregate maps $\Phi_n$, as random elements of $\Sigma$. 

The following result is well-known, but we give a proof for completeness. (With additional work, one could obtain estimates on rates of convergence. We do not pursue this direction here, however see Remark 3 after Lemma \ref{ucomp}.)
\begin{Pro}\label{loewnercontprop}
Let $T>0$ be given. For $j=1,2$ let $\Psi_{t}^{(j)}, \, 0 \le t \le T,$ be the solution to the Loewner equation \eqref{loewnerPDE} with driving term $\xi_{t}^{(j)}$. For every $\ee > 0$ there exists $\delta=\delta(\ee,T)>0$ such that if $\|e^{i\xi^{(1)}}-e^{i\xi^{(2)}}\|_{\infty} < \delta$, then
\[
\sup_{0 \le t \le T} \sup_{\{|z| \ge 1+\ee\}}\left|\Psi^{(1)}_{t}(z) - \Psi^{(2)}_{t}(z) \right| < \ee.
\]
%
\end{Pro}
\begin{proof}
Fix $s \in [0,T]$ and consider the reverse-time Loewner equation \eqref{loewnerODE}. We let $u^{(j)}_{t}$ be the reverse flow driven by $\xi^{(j)}_{s-t}$ for $0 \le t \le s$. Write $W_{t}^{(j)} = e^{i\xi^{(j)}_{s-t}}$. Taking the difference and differentiating $H=u^{(1)}-u^{(2)}$ with respect to $t$ gives
\[
\dot H - H v= (W^{(1)}-W^{(2)}) w, 
\]
where
\[
v = v(t)= \frac{u^{(1)}u^{(2)}-W^{(1)}W^{(2)} - (1/2)(u^{(1)}+u^{(2)})(W^{(1)} + W^{(2)})}{(u^{(1)}-W^{(1)})(u^{(2)}-W^{(2)})}
\]
and
\[
w=w(t) = \frac{(u^{(1)}+u^{(2)})^{2} }{2(u^{(1)}-W^{(1)})(u^{(2)}-W^{(2)})}.
\]
Since the flows move away from the unit circle, these expressions show that there is a constant $A$ depending only on $T$ such that if $|z| \ge 1+\ee$ then for all $0 \le t \le s \le T$,
\[
\RRe v(t) \le A/\ee^{2}, \qquad |w(t)| \le A/\ee^{2}.
\]
Since $H(0)=0$,
\[
H(t) = \int_{0}^{t}\left[e^{\int_{s}^{t}v(r)dr} (W^{(1)}(s)-W^{(2)}(s)) w(s) \right]\dx s
\]
and consequently, for a different $T$-dependent $A$,
\[
\sup_{\{|z| \ge 1+\ee \}}|\Psi^{(1)}_{t}(z) -\Psi^{(2)}_{t}(z)|=\sup_{\{|z|\ge 1+\ee\}} |H(t)| \le \|W^{(2)}-W^{(1)}\|_{\infty}e^{A/\ee^{2}}A/\ee^{2}.
\]
Hence we can take $\delta< e^{-A/\ee^{2}}\ee^{3}/A$ and this is clearly uniform in $0 \le t  \le T$.
\end{proof}

Thus, we obtain convergence in law of conformal maps provided we can show convergence in law of driving processes. Note that in our main result we have convergence to a degenerate deterministic limit (modulo rotation). As is explained in \cite[Section 4.2]{JS09}, we 
can strengthen the convergence that follows from Proposition \ref{loewnercontprop} in this instance, and obtain convergence of $K_n$ with respect to the Hausdorff metric in $\Delta$.

\section{Analysis of the slit map}\label{slitder}
In our arguments, we shall need effective bounds on the derivative $f'_t(z)$ of the slit map, in order to estimate moments of angle sequences, among other things. 
An explicit formula for the slit map $f_t\colon \Delta \to \Delta\setminus(1,1+d(t)]$ was given in \eqref{slitexplicit}, while the length $d(t)$ of the growing slit is given by 
\eqref{sticklength}. We note that $f_t(1)=1+d(t)$, and that one can compute that $f_t(e^{i\beta_t})=f_t(e^{-i\beta_t})=1$ for 
\begin{equation}
\beta_t=2\arctan\left(\frac{d(t)}{2\sqrt{d(t)+1}}\right).
\label{basepointformula}
\end{equation}
We shall refer to $\exp(i\beta_t)$ and $\exp(-i\beta_t)$ as the base points of the slit. In our scaling limit results, we will make use of the facts that 
\begin{equation}
\frac{\beta_t}{d(t)}\to 1\quad \textrm{and}\quad \frac{d(t)}{2t^{1/2}} \to 1, \quad \textrm{as} \quad t \to 0.
\label{basepointasymps}
\end{equation}

\subsection{Pointwise estimates}
We begin by obtaining bounds on the (spatial) derivative of the slit map $f_t(z)$. To get a feeling for the overall behavior of these derivatives, it is instructive to first compute the derivative of the half-plane slit map, 
\[F_t'(z)=\frac{z}{(z^2-4t)^{1/2}}.\]
From this formula, it is apparent that $F_t'(z)$ has a zero at the point that is mapped to the tip of the slit, and square-root type singularities at points mapping to the base of the slit.
We show that the slit map in the exterior disk exhibits the same type of local behavior.
\begin{lemma}
\label{mobiusdecomp}
For all $t>0$ and $|z|>1$, we have
\begin{equation}
f'_t(z)=H_t(z)\frac{z-1}{\left(z-e^{i\beta_t}\right)^{1/2}\left(z-e^{-i\beta_t}\right)^{1/2}}
\label{slitdersimple}
\end{equation}
where $H_t(z)$ is holomorphic in $z$, has $\lim_{z\to \infty} H_t(z)=e^{t}$, and satisfies 
\[1 \leq |H_t(z)|\leq 4 e^t.\]
\end{lemma}
\begin{proof}
Since the slit map $f_t(z)$ solves the Loewner equation 
\[\partial_tf_t(z)=zf'_t(z)\frac{z+1}{z-1}\]
we have
\begin{equation}
f'_t(z)=\frac{z-1}{z(z+1)}\partial_tf_t(z).
\label{slitwrtt}
\end{equation}
Differentiating the explicit expression \eqref{slitexplicit} with respect to $t$, we find that
\[\partial_tf_t(z)=\frac{e^t}{2z}\frac{z+1}{\sqrt{(z+1)^2-4e^{-t}z}}\left((z+1)\sqrt{(z+1)^2-4e^{-t}z}+(z+1)^2-2e^{-t}z\right).\]
Inserting this into \eqref{slitwrtt}, we obtain
\[f'_t(z)=H_t(z)\frac{z-1}{\sqrt{(z+1)^2-4e^{-t}z}}\]
with
\[H_t(z)=\frac{e^{t}}{2z^2}\left[(z+1)\left(z+1+\sqrt{(z+1)^2-4e^{-t}z}\right)-2e^{-t}z\right].\]
It remains to show that $H_t(z)$ is bounded above and below. But this follows immediately upon writing $H_t(z)=z^{-1}f_t(z)$, where $f_t(z)$ is the 
slit map itself, and observing that $1 \leq |f_t(z)|/|z| \leq (1+ d(t)) \vee e^t \leq 4e^t$. 
Finally, we verify that $z_t=e^{i\beta_t}$ solves $(z+1)^2-4e^{-t}z=0$, and this leads to the factorization $(z+1)^2-4e^{-t}z=(z-e^{i\beta_t})(z-e^{-i\beta_t})$.
\end{proof}
\begin{figure}[ht!]
\begin{center}
\includegraphics[width=8cm]{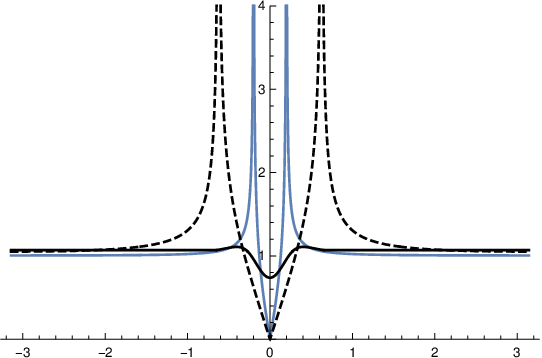}
\includegraphics[width=8cm]{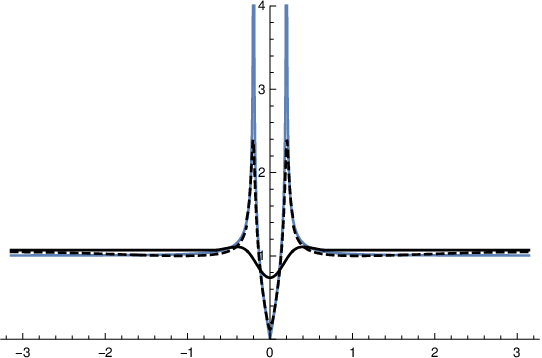}
\end{center}
  \caption{\textsl{Plots of $\theta\mapsto|f'_t(e^{\parsig+i\theta})|$. Left: $\parsig=0.0001$ fixed, $t=0.01$ (blue) and $t=0.1$ (dashed). Right: $t=0.01$ fixed,  $\parsig=0.0001$ (blue) and $\parsig=0.02$ (dashed)}. \newline Plot with $t=0.01$ and $\parsig=0.2$ (black) shown in both pictures for comparison.}
\label{slitderfig}
\end{figure}
Our analysis of the $\mathrm{ALE}$ model will require local estimates on the derivative of the slit map. Representative graphs of how $\theta\mapsto|f'_t(e^{\parsig+i\theta})|$ varies with $t$ and $\parsig$ are shown in Figure \ref{slitderfig}.

\begin{lemma}\label{slitmapderivativebds}
Fix $T>0$, let $0<t\leq T$ and suppose $|z|-1\leq d(t)$. Then the derivative of the slit map admits the following estimates, where $A_1$ and $A_2$ are non-zero constants depending only on $T$:
\begin{enumerate}
\item (Near the tip) For $|\arg z|<\frac{1}{2}\beta_t$, 
\begin{equation*}
A_1\frac{|z-1|}{d(t)}\leq |f'_{t}(z)|\leq A_2\frac{|z-1|}{d(t)}.
\end{equation*} 
\item (Near the base) For $|\arg z \pm \beta_t|\leq \frac{1}{2}\beta_t$,
\[A_1 \leq |f'_{t}(z)|\leq A_2 \frac{d(t)}{|z|-1}.\]
\item (Away from tip and base) For $\frac{3}{2}\beta_t < |\arg z| \leq \pi$,
\[A_1 \leq |f'_{t}(z)|\leq A_2.\]
\end{enumerate}

\end{lemma}
\begin{proof}
We treat the case $|\arg z|<\frac{1}{2}\beta_t$ first. In light of the global bounds on the function $H_t(z)$ from Lemma \ref{mobiusdecomp}, it suffices to estimate the square root expressions appearing in the denominator in \eqref{slitdersimple}.
We have
\begin{equation}
|z-e^{i\beta_t}|=|e^{\log|z|+i(\arg{z}-\beta_t)}-1|\asymp \left((\log|z|)^2+(\arg(z)-\beta_t)^2\right)^{1/2}\asymp \log |z| \vee d(t).
\label{auxmodelinsert}
\end{equation}
If $0<|z|-1\leq d(t)$ this yields 
\[|z-e^{i\beta_t}|^{1/2}|z-e^{-i\beta_t}|^{1/2}\asymp d(t),\]
as claimed.

Near the base, the same reasoning as before shows that
$|z-1|\asymp d(t)$.
On the other hand,
\[|z|-1\leq |z-e^{i\beta_t}|\leq |e^{\log|z|+i(\beta_t+\frac{1}{2}\beta_t)}-e^{i\beta_t}|\leq Ad(t),\]
where the lower bound is attained when $\arg(z)=\beta_t$. 
Combining these bounds leads to the claimed estimates for $|\arg z \pm \beta_t|\leq \frac{\beta_t}{2}$.

On each fixed radius, the function $v\colon \arg(z)\mapsto \left|\frac{z-1}{(z-e^{i\beta_t})^{1/2}(z-e^{-i\beta_t})^{1/2}}\right|$ is decreasing on $[\frac{3}{2}\beta_t,\pi]$, with $v(\pi)=(e^{\log|z|}+1)/((e^{\log|z|}+\cos\beta_t)^2+\sin^2\beta_t)^{1/2}\geq 1$. So in order to obtain the last set of estimates, it suffices to note that $v$
remains bounded above and below as $\arg(z) \to \frac{3}{2}\beta_t$, by the same arguments as before.
\end{proof}

\subsection{Moment computations}

We now return to random growth models and present the moment bounds that will be needed in Section \ref{slitconvergencesection}. 
As before, $\parsig>0$ is our regularization parameter, while $\eta>0$ is a model parameter. 

Define the normalization factor
\begin{equation}
Z^*_t=Z^*_t(\eta,\parsig)=\int_{\mathbb{T}}|f'_{t}(e^{\parsig+is})|^{-\eta}\rm{\dx} s.
\label{normfactor}
\end{equation}

 \begin{lemma}\label{slitmapvariance}
Fix $T>0$ and $\eta \geq 0$. There exist constants $A_1$ and $A_2$ depending only on $T$ and $\eta$ such that, for all
$0<t\leq T$, the total mass $Z_t^*$ satisfies the following. 
\begin{itemize}
\item ($\eta< 1$) For all $\parsig > 0$,
\begin{equation}
A_1\leq Z^*_t\leq A_2.
\end{equation}
In particular, $Z^*_t$ remains finite as $\parsig\to 0$.
\item ($\eta >1$)
For all $0<\parsig\leq t^{\frac{\eta}{2(\eta-1)}}$,
\begin{equation}
A_1d(t)^{\eta}\parsig^{-(\eta-1)}\leq Z^*_t\leq A_2d(t)^{\eta}\parsig^{-(\eta-1)}. 
\label{slitintegral}
\end{equation}
In particular, $Z^*_t$ diverges as $\parsig \to 0$ with $\parsig \ll t^{\frac{\eta}{2(\eta-1)}}$.
\end{itemize}

Moreover, for $\eta>1$ and $0<\parsig\leq t^{\frac{\eta}{2(\eta-1)}}$ we have the following estimates: 
\begin{enumerate}
\item (Near the tip) For $|\theta|<\frac{\beta_t}{2}$,
\begin{equation*}
A_1\frac{1}{\parsig}\left(1+\left(\frac{\theta}{\parsig}\right)^2\right)^{-\eta/2}\leq \frac{1}{Z^*_t}|f'_{t}(e^{\parsig+i\theta})|^{-\eta}\leq A_2\frac{1}{\parsig}\left(1+\left(\frac{\theta}{\parsig}\right)^2\right)^{-\eta/2}.
\end{equation*} \item (Near the base) For $|\theta-\beta_t|\leq \frac{1}{2}\beta_t$,
\[A_1\parsig^{2\eta-1}d(t)^{-2\eta}\leq \frac{1}{Z^*_t}|f'_{t}(e^{\parsig+i\theta})|^{-\eta}\leq A_2 \parsig^{\eta-1}d(t)^{-\eta}.\]
\item (Away from the tip and base) For $\frac{3}{2}\beta_t < |\theta| \leq \pi$,
\[A_1\parsig^{\eta-1}d(t)^{-\eta}\leq \frac{1} {Z^*_t}|f'_{t}(e^{\parsig+i\theta})|^{-\eta}\leq A_2\parsig^{\eta-1}d(t)^{-\eta}.\]
\end{enumerate}
\end{lemma}
\begin{proof}
We begin by treating the case $\eta<1$. In light of Lemma \ref{slitmapderivativebds}, non-trivial global bounds on $Z_t^*$ from above and below follow immediately from the bounds on $|f_t'(e^{\parsig + is})|$ for $|s|>\frac{3}{2}\beta_t$ provided the contribution from $(-\frac{\beta_t}{2}, \frac{\beta_t}{2})$ is finite.
Hence it suffices to estimate the integral
\[\int_{-\frac{\beta_t}{2}}^{\frac{\beta_t}{2}}|f_t'(e^{\parsig+is})|^{-\eta}\dx s\asymp Ad(t)\int_0^{\frac{\beta_t}{2}}\frac{1}{(\parsig^2+s^2)^{\eta/2}}\dx s,\]
where we have used that $A_1<|e^{\parsig+is}-1|/(\parsig^2+s^2)^{1/2}<A_2$ for $s, \parsig$ small. Next, we note that
\[\int_0^{\frac{\beta_t}{2}}\frac{1}{(\parsig^2+s^2)^{\eta/2}}\dx s\leq \int_{0}^{\frac{\beta_t}{2}}\frac{1}{s^{\eta}}\dx s,\]
and the latter integral is bounded for $0<t<T$ since $\eta<1$.

We turn to the case $\eta>1$.  Since the integral $\int |f_{t}'(e^{\parsig+is})|^{-\eta}\dx s$ now diverges as $\parsig \to 0$ due to the singularity at $s=0$, it again suffices to estimate the contribution coming from $|s|<\beta_t/2$ in order to establish \eqref{slitintegral}. We have
\begin{align*}\int_0^{\frac{\beta_t}{2}}|f_t(e^{\parsig+is})|^{-\eta}\dx s&\leq Ad(t)^{\eta}\int_0^{\frac{\beta_t}{2}}(\parsig^2+s^2)^{-\eta/2}\dx s\\&
\leq Ad(t)^{\eta}\parsig^{-\eta}\int_0^{\frac{\beta_t}{2\parsig}}
\parsig(1+u^2)^{-\eta/2}\dx u
\end{align*}
after a change of variables. Since $\int_0^{\infty}(1+u^2)^{-\eta/2}\dx u$ is now finite, the upper bound follows. Similar reasoning together with the assumption that $\parsig\leq t^{1/2}$ yields the lower bound on the integral. The estimates on the normalized derivative follow upon dividing through by $Z_t^*$ in Lemma \ref{slitmapderivativebds}.
\end{proof}

We now turn to moment bounds for $\eta>1$.
\begin{lemma}\label{slitmapmoments}\label{slitmapmomentbounds}
For all $\eta$ and $\parsig>0$,
\[\int_{-\pi}^\pi \theta \,\frac{1}{Z_t^*}|f'_{t}(e^{\parsig+i\theta})|^{-\eta}\dx\theta=0.\]

Now suppose $\eta>1$ and $\parsig$ satisfies the hypotheses of Lemma \ref{slitmapvariance}.
Let $x\in (\parsig, \frac{\beta_t}{2})$. Then, for $1<\eta< 3$, we have
\[A_1x^{3-\eta}\parsig^{\eta-1}\leq \int_{-x}^{x}\theta^2\frac{1}{Z^*_t}|f_t'(e^{\parsig+i\theta})|^{-\eta}\dx \theta\leq A_2x^{3-\eta}\parsig^{\eta-1},\]
and for $\eta=3$, we have
\[A_1\parsig^2\log(x\parsig^{-1})\leq \int_{-x}^{x}\theta^2\frac{1}{Z^*_t}|f_t'(e^{\parsig+i\theta})|^{-\eta}\dx \theta \leq A_2\parsig^2\log(x\parsig^{-1}).\]

For $\eta>3$, we have
\[A_1\parsig^2\leq \int_{-x}^x\theta^2\, \frac{1}{Z_t^*}|f'_{t}(e^{\parsig+i\theta})|^{-\eta}\dx\theta\leq A_2\parsig^2.\]

Under the same assumptions as above,
for $1<\eta< 2$, we have
\[A_1x^{2-\eta}\parsig^{\eta-1}\leq \int_{-x}^{x}|\theta|\frac{1}{Z^*_t}|f_t'(e^{\parsig+i\theta})|^{-\eta}\dx \theta\leq A_2x^{2-\eta}\parsig^{\eta-1},\]
and for $\eta=2$,
\[A_1\parsig \log(x\parsig^{-1})\leq \int_{-x}^{x}|\theta|\frac{1}{Z^*_t}|f_t'(e^{\parsig+i\theta})|^{-\eta}\dx \theta \leq A_2\parsig \log(x\parsig^{-1}).\]
Finally, for $\eta>2$,
\[A_1\parsig \leq \int_{-x}^x|\theta|\, \frac{1}{Z_t^*}|f'_{t}(e^{\parsig+i\theta})|^{-\eta}\dx\theta\leq A_2\parsig .\]
\end{lemma}
\begin{proof}
The statement that $\int \theta|f'_t(e^{\parsig+i\theta})|^{-\eta}\dx\theta=0$ follows immediately from symmetry of the function $\theta\mapsto |f'_t(e^{\parsig+i\theta})|$ for each $\parsig$ and $ t$. 

We turn to second moments, and deal with the parameter range $1<\eta\leq 3$ first. 
By Lemma \ref{slitmapvariance},
\[\int_{-x}^{x}\theta^2\frac{1}{Z_t^*}|f_t'(e^{\parsig+i\theta})|^{-\eta}\dx\theta=2\int_0^{x}\theta^2\frac{1}{Z_t^*}|f_t'(e^{\parsig+i\theta})|^{-\eta}\dx\theta\asymp \parsig^2\int_0^{x}\frac{\left(\frac{\theta}{\parsig}\right)^2}{\left(1+(\frac{\theta}{\parsig})^2\right)^{\eta/2}}\frac{\dx \theta}{\parsig}.\]
Performing a change of variables, and assuming  $\eta< 3$, we obtain the integral
\begin{align*}
\parsig^2\int_0^{\frac{x}{\parsig}}u^2(1+u^2)^{-\eta/2}\dx u&=\parsig^2\int_0^1u^2(1+u^2)^{-\eta/2}\dx u+\parsig^2\int_1^{\frac{x}{\parsig}}u^2(1+u^2)^{-\eta/2}\dx u\\
&\asymp A_1\parsig^2+A_2\parsig^2\int_1^{\frac{x}{\parsig}}u^{2-\eta}\dx u \\
&=A_1\parsig^2+A_2\frac{1}{3-\eta}\parsig^2x^{3-\eta}\parsig^{\eta-3}\\
&=A_1\parsig^2+A_2x^{3-\eta}\parsig^{\eta-1},
\end{align*}
as claimed. An obvious modification of the argument leads to bounds for $\eta=3$.

Finally, we treat the case $\eta>3$ and show that the second moment decays like $\parsig^2$ independently of  $\eta$. 
It now suffices to examine
\[\int_{|\theta|<x}\theta^2\,\frac{1}{Z_t}|f_t'(e^{\parsig+i\theta})|^{-\eta}\dx\theta \asymp2\parsig^2\int_{0}^{\frac{x}{\parsig}}u^2(1+u^2)^{-\eta/2}\dx u.\]
The integral on the right now converges since $\eta>3$, and in fact 
\[\int_0^{\infty}u^2(1+u^2)^{-\eta/2}\dx u=\frac{\sqrt{\pi}}{4}\frac{\Gamma(\frac{\eta-3}{2})}{\Gamma(\frac{\eta}2)}.\]
To get the lower bound, 
we use the assumption $1<x/\parsig$ to bound the integral from below. The second assertion of the Lemma follows.

Analogous calculations lead to the quoted bounds on the first moments.
\end{proof}

\section{Ancestral lines and convergence for $\mathrm{ALE}$}\label{slitconvergencesection}
We now present a proof of our main convergence theorem, conditional on technical results proved in the final section of the paper, and discuss possible extensions of our results. 
\subsection{Convergence in the Markovian model}\label{Markovproof}
We first prove Theorem \ref{thm: auxslitconv}, which we restate for the reader's convenience. Recall that $K_{n(t)}^*$ is the cluster associated with $\Phi^*_{n(t)}$, and the event
\[
\Omega_N^* = \{ \mbox{Particle $j$ in the $\ast$-model has parent $j-1$ for all } j=1, \dots, N \}.
\]

\begin{customthm}{\ref{thm: auxslitconv}}
Set $\sigma_0 = \capc^{\gamma^*}$ where
\[
\gamma^* > \frac{\eta+1}{2(\eta-1)}.
\]
Then
\begin{align*}
\lim_{\capc \to 0} \inf_{0<\parsig<\sigma_0} \mathbb{P}(\Omega_N^*) &= 1 \quad \mbox{if } \eta > 1 \\
\lim_{\capc \to 0} \sup_{\parsig>0} \mathbb{P}(\Omega_N^*) &=0 \quad \mbox{if } \eta<1.  \label{clustersfig}
\end{align*} 
Furthermore, when $\eta > 1$ and $\parsig<\sigma_0$, for any $r>1$ and $T<\infty$,
\[
\sup_{t \leq T} \sup_{\{|z|>r\}}|\Phi^*_{n(t)}(z) - e^{i\theta_1^*}f_t(e^{-i\theta_1^*}z) | \to 0 \quad \textrm{ in probability as }\quad \capc\to 0,
\]
and the cluster $K^*_{n(t)}$ converges in the Hausdorff topology to a disk with slit of logarithmic capacity $t$ attached at position $z=e^{i\theta_1^*}$.
\end{customthm}

\begin{proof}
Since we can always rotate the clusters $K_n^*$ by a fixed angle, without loss of generality, we assume that the initial angle $\theta_1^*=0$.
As explained in Section \ref{overview}, we choose to sample $\theta_k^*$ from the interval $[\theta_{k-1}^*-\pi, \theta_{k-1}^*+\pi)$. This means that we can write $\theta_n^* = u_2+\cdots+u_n$ where the $u_k$ are independent $[-\pi,\pi)$-valued random variables and $u_k=\theta_k^*-\theta_{k-1}^*$ has symmetric distribution $h_k^*(\theta|0)$. 

First suppose $\eta>1$. Then by \eqref{basepointasymps} and Lemma \ref{slitmapvariance} there exists some constant $A$ (which may change from line to line), depending only on $T$ and $\eta$, such that for all $k\leq N$,
\[
A^{-1} (k\capc)^{1/2} < \beta_{k\capc} < A(k\capc)^{1/2},
\]
\[
\frac{A^{-1}}{\parsig} \left ( 1 + \frac{\theta^2}{\parsig^2} \right )^{-\eta/2} \leq h^*_k(\theta|0) \leq \frac{A}{\parsig} \left ( 1 + \frac{\theta^2}{\parsig^2} \right )^{-\eta/2}
\quad \mbox{for $|\theta| < \frac{\beta_{k\capc}}{2}$,}\]
and 
\[
h^*_k(\theta|0) \leq A \parsig^{\eta-1} (\capc k)^{-\eta/2} \quad \mbox{for $|\theta|>\frac{\beta_{k\capc}}{2}$.}
\]
Therefore
\[
\mathbb{P}\left(|u_k|\geq \frac{\beta_{\capc}}{2}\right) =2\int_{\frac{\beta_{\capc}}{2}}^{\frac{\beta_{k\capc}}{2}}h^*_{k}(\theta|0)\dx \theta
+2\int_{\frac{\beta_{k\capc}}{2}}^{\pi}h^*_{k}(\theta|0)\dx \theta\leq A (\parsig^{\eta-1}\capc^{\frac{1}{2}(1-\eta)}+\parsig^{\eta-1} (\capc k)^{-\eta/2}).
\]
Hence, for $\eta>1$,
\begin{equation*}
\mathbb{P}((\Omega_N^*)^c) \leq \mathbb{P}\left(\sup_{2 \leq k \leq N}|\theta^*_k - \theta^*_{k-1}| \geq \frac{\beta_{\capc}}{2}\right) 
\leq \sum_{k=2}^N \mathbb{P}\left(|u_k| \geq \frac{\beta_{\capc}}{2}\right) 
\leq A \parsig^{\eta-1} \capc^{-\frac{1}{2}(\eta-1)}\capc^{-1}\longrightarrow 0
\end{equation*}
as $\capc \to 0$. 

Now suppose that $\eta<1$ and $\parsig \to 0$ as $\capc \to 0$. 
Using Lemma \ref{slitmapderivativebds} and setting $|z|=e^{\parsig}$ in \eqref{auxmodelinsert}, and then letting $\capc\to 0$, we get
\begin{equation*}
\mathbb{P}(\Omega_N^*) \leq \mathbb{P}\left(|\theta_2^*| < \beta_{\capc}\right) 
\leq A \left( \int_0^{\frac{\beta_{\capc}}{2}} \frac{\capc^{\eta/2}\vee \parsig^{\eta}}{(\parsig^2 + \theta^2)^{\eta/2}}\dx \theta+\int_{\frac{\beta_{\capc}}{2}}^{\beta_{\capc}}\dx \theta\right) 
\leq A \capc^{1/2} \vee \parsig^{\eta} \longrightarrow 0.
\end{equation*}
If $\parsig$ is bounded below, then $h_k^*$ is uniformly bounded above and below, and $\PP(|u_k|\leq \beta_{\capc})=2\int_0^{\beta_{\capc}}h_k^*(\theta|0)d\theta \to 0$ since $\beta_{\capc}\to 0$ with $\capc$.

To show convergence of $\Phi_{n(t)}^*(z)$ to $f_t(z)$ for $t<T$ when $\eta>1$ and $\parsig<\sigma_0$, by Proposition \ref{loewnercontprop} it is enough to show that $\sup_{n \leq N}|\theta^*_n| \to 0$
with high probability as $\capc \to 0$.
To do this, we write
\[
\theta^*_n = \sum_{k=2}^n u_k \mathbf{1}_{\{|u_k| < \beta_{\capc}/2\}} + \sum_{k=2}^n u_k \mathbf{1}_{\{|u_k| \geq \beta_{\capc}/2\}}
\]
and note that $M^*_n = \sum_{k=2}^n u_k \mathbf{1}_{\{|u_k| < \beta_{\capc}/2\}}$ is a martingale. By the same argument as used to show $\mathbb{P}((\Omega_N^*)^c) \to 0$,
\[
\mathbb{P} \left ( \theta_n^*=M_n^* \mbox{ for all } n \leq N \right ) \geq 1 - A \parsig^{\eta-1} \capc^{-\frac{1}{2}(\eta-1)}\capc^{-1} \to 1.
\]
Convergence of $\sup_{n\leq N}|\theta^*_n|$ to $0$ follows from moment bounds in Lemma \ref{slitmapmoments} together with standard martingale arguments (viz. the proof of Theorem \ref{slitstheorem}).   
\end{proof}

\subsection{The ancestral lines and convergence theorem}
We now return to the $\mathrm{ALE}(0,\eta)$ process and show how the bounds obtained above, together with certain comparison results that will be proved in the next section, allow us to prove the analogue of Theorem \ref{thm: auxslitconv} for
 the $\Phi_n$ maps that generate $\mathrm{ALE}(0,\eta)$ clusters.

Without loss of generality we may set $\theta_1=0$. Let
\begin{equation}
\label{hdef}
h_k(\theta)=\frac{1}{Z_k}|\Phi_{k-1}'(e^{\parsig+i \theta})|^{-\eta}, \quad k=2, 3, \ldots
\end{equation}
denote the density functions conditional on $\mathcal{F}_{k-1}$ associated with the angle sequence $\{\theta_k\}$ of the $\mathrm{ALE}(0, \eta)$-model with model parameter $\eta \in  \RR$, particle capacity parameter $\capc \in (0,1)$ and regularization parameter $\parsig \in (0,1)$. 
As usual, let $\mathcal{F}_k$ be the $\sigma-$algebra generated by $\theta_1, \dots, \theta_k$.

We first state a precise estimate for how well $|\Phi'_n(e^{\parsig + i \theta})|$ can be approximated by $|(f^{\theta_{n}}_{n \capc})'(e^{\parsig + i \theta})|$. In Section \ref{overview}, we discussed how the intermediate particles are visible in the derivative of $\Phi_n(z)$ in a way they are not in $f_{n \capc}^{\theta_n}(z)$ (see Figure \ref{singularities}). The estimates below capture this discrepancy.  

\begin{lemma}
\label{phinbounds}
Fix $T>0$, let $n \leq \lfloor T/\capc \rfloor$ and set $\epsilon_n=(e^\parsig - 1) \vee \sup_{k\leq n} |\theta_k|$. 
\begin{itemize}
\item[(i)] 
There exists some absolute constant $A>1$, such that if $|\theta-\theta_n|<\capc^{1/2}$ and $\epsilon_n  < A^{-1} \capc^{1/2}$, then
\begin{equation}
\label{rateq}
 \left| \left| \frac{\Phi'_n(e^{\parsig + i \theta})}{(f^{\theta_{n}}_{n \capc})'(e^{\parsig + i \theta})} \right | - 1 \right | < A \epsilon_n^2 \capc^{-1}.
\end{equation}
\item[(ii)] There exist absolute constants $A$ and $B$ only dependent on $T$, such that  if $\epsilon_n \leq A^{-1} \capc^{1/2}$, then 
\[
\left | \Phi'_n(e^{\parsig + i \theta}) \right | \geq B^{-1} \epsilon_n^{-1} \parsig (1 - \cos(\theta-\theta_n))^{1/2}.
\]
\end{itemize}
\end{lemma}
The proof of Lemma \ref{phinbounds} relies on a refined analysis of solutions to the Loewner equation in the case where driving functions are uniformly close, and will be presented in Section \ref{conformal} to avoid interrupting the flow of the proof of the main theorem below.

We now prove our main result. For fixed $T>0$, set $N=\lfloor T/\capc \rfloor$. Recall the definition of $\Omega_N$ from Section \ref{overview}.

\begin{theorem}\label{slitstheorem}
Set $\sigma_0 = \capc^{\gamma}$ for
\[ 
\gamma > \frac{5}{4} \vee \frac{(2\lambda+1)\eta+1}{2(\eta-1)}, 
\]
where
\[ 
\lambda = \lambda(\eta) = 
\begin{cases} 
 \frac{1}{\eta-1} &\mbox{ if } 1< \eta < 3; \\
 \frac{1}{2} & \mbox{ if }\eta \geq 3.
\end{cases} 
\]

Then, for all $T < \infty$,
\begin{align*}
\lim_{\capc \to 0} \inf_{0<\parsig<\sigma_0} \mathbb{P}(\Omega_N) &= 1 \quad \mbox{if } \eta > 1 \\
\lim_{\capc \to 0} \sup_{\parsig>0} \mathbb{P}(\Omega_N) &=0 \quad \mbox{if } \eta<1. 
\end{align*} 

Furthermore, when $\eta > 1$ and $\parsig<\sigma_0$, for any $r>1$ and $T < \infty$,
\[
\sup_{t \leq T} \sup_{|z|>r}|\Phi_{n(t)}(z) - f_t(z) | \to 0 \quad \textrm{ in probability as }\quad \capc\to 0,
\]
and hence the cluster $K_{n(t)}$ converges in the Hausdorff topology to a disk with slit of logarithmic capacity $t$ attached at position $1$.
\end{theorem}
\begin{proof}
Fix $\eta>1$ and let
\begin{equation}
N_T=\inf\left\{k\geq 1\colon |\theta_k|> \parsig k^{\lambda} (\log \capc^{-1})^{6 \lambda} \right\}\wedge N.
\end{equation} 
Observe that, since $\parsig < \sigma_0$, we have
\[
\parsig n^{\lambda} (\log \capc^{-1})^{6 \lambda} \leq \left ( T^{\lambda} \capc^{\gamma - (\lambda+1/2)} (\log \capc^{-1})^{6 \lambda} \right ) \capc^{1/2}.
\]
Hence, using the fact that $\gamma>\lambda+1/2$, and that $A^{-1} \capc^{1/2} \leq \beta_\capc \leq A \capc^{1/2}$, there exists some $c_0 > 0$, dependent only on $T$ and $\eta$, such that if $\capc < c_0$, then $ \{ N_T=N \} \subseteq \Omega_N$. From now on assume that $\capc < c_0$.
We shall  prove that $\mathbb{P}(N_T = N) \to 1$ as $\capc \to 0$. Once this has been done, it follows that if $\eta>1$,
\[
\lim_{\capc \to 0} \inf_{0<\parsig<\sigma_0} \mathbb{P}(\Omega_N) = 1.
\]
Exactly the same argument as Theorem \ref{thm: auxslitconv} can then be used to show that 
\[
\lim_{\capc \to 0} \sup_{\parsig>0} \mathbb{P}(\Omega_N) =0 
\]
if $\eta < 1$, and that when $\eta > 1$ and $\parsig<\sigma_0$, for any $r>1$ and $T<\infty$,
\[
\sup_{t \leq T} \sup_{|z|>r}|\Phi_{n(t)}(z) - f_t(z) | \to 0 \quad \mbox{ in probability as }\quad \capc\to 0,
\]
and hence the cluster $K_{n(t)}$ converges in the Hausdorff topology to a disk with slit of logarithmic capacity $t$ attached at $1$.

We turn to the proof. Suppose that $n < N_T$. As before, using the fact that $\parsig < \sigma_0$, we have
\[
\epsilon_n \leq \parsig n^{\lambda} (\log \capc^{-1})^{6 \lambda} \leq \left ( T^{\lambda} \capc^{\gamma - (\lambda+1/2)} (\log \capc^{-1})^{6 \lambda} \right ) \capc^{1/2},
\]
where $\epsilon_n=(e^{\parsig}-1)\vee \sup_{k\leq n}|\theta_k|$ as in Lemma \ref{phinbounds}.
Hence there exists some $0< c_1 < c_0$, dependent only on $T$ and $\eta$, such that if $\capc < c_1$, then $\epsilon_n$ satisfies the conditions of Lemma \ref{phinbounds}. From now on assume that $\capc < c_1$. Then, by Lemma \ref{phinbounds}, there exists $A_n$ such that, if $|\theta-\theta_n| \leq \capc^{1/2}$
\[
(1-A_n)|f_{n\capc}'(e^{\parsig + i(\theta-\theta_n)})|^{-\eta}<|\Phi'_n(e^{\parsig + i\theta})|^{-\eta}<(1+A_n)|f_{n\capc}'(e^{\parsig + i(\theta-\theta_n)})|^{-\eta},
\]
and furthermore $A_n = A_\eta \parsig^2 \capc^{-1} n^{2\lambda} (\log \capc^{-1})^{12 \lambda}$ for $A_{\eta}$ that depends only on $\eta$ and $T$.

We begin by getting estimates on 
\[
Z_n = \int_{\mathbb{T}} |\Phi'_n(e^{\parsig + i\theta})|^{-\eta} \dx \theta. 
\]
We have
\begin{align*}
\int_{\mathbb{T}} |\Phi'_n(e^{\parsig + i\theta})|^{-\eta} \mathbf{1}_{\{\capc^{1/2} < |\theta-\theta_n| < \pi \} }  \dx \theta
&\leq 2 B^\eta  n^{\lambda \eta} (\log \capc^{-1})^{6 \lambda \eta} \int_{\capc^{1/2}}^\pi (1-\cos u)^{-\eta/2} \dx u \\
&\leq B' n^{\lambda \eta} \capc^{-(\eta - 1)/2} (\log \capc^{-1})^{6 \lambda \eta}
\end{align*}
for some $B'$ that depends only on $\eta$ and $T$.
Using the notation of Section \ref{overview}, recall from Lemma \ref{slitmapvariance} that there exist $A', A''$ depending only on $\eta$ and $T$ such that
\[
A' (n \capc)^{\eta/2} \parsig^{-(\eta-1)} \leq Z^*_{n \capc} \leq A'' (n \capc)^{\eta/2} \parsig^{-(\eta-1)}.
\]
Hence, 
\[
(Z_{n \capc}^*)^{-1} \int_{\mathbb{T}} |\Phi'_n(e^{\parsig + i\theta})|^{-\eta} \mathbf{1}_{\{\capc^{1/2} < |\theta-\theta_n| < \pi \} }  \dx \theta \leq B_\eta \parsig^{\eta-1} n^{(\lambda -1/2) \eta} \capc^{-(2 \eta - 1)/2} (\log \capc^{-1})^{6 \lambda \eta}
\]
for some $B_\eta$ that depends only on $\eta$ and $T$. Set $B_n=B_\eta \parsig^{\eta-1} n^{(\lambda -1/2) \eta} \capc^{-(2 \eta - 1)/2} (\log \capc^{-1})^{6 \lambda \eta}$. 

Observe that the choice of $\gamma$ ensures that, provided $\parsig < \sigma_0$, we have $N^{(1-\lambda) \vee 0}A_N\to 0$ and $NB_N \to 0$. We shall see that these conditions are sufficient to prove our result.

Now
\begin{align*}
Z_n &= \int_{\mathbb{T}} |\Phi'_n(e^{\parsig + i\theta})|^{-\eta} \left ( \mathbf{1}_{\{|\theta-\theta_n|\leq\capc^{1/2}\}} + \mathbf{1}_{\{\capc^{1/2} < |\theta-\theta_n| < \pi \} } \right ) \dx \theta \\
&\leq 2 (1+A_n) \int_{0}^{\capc^{1/2}} |f_{n\capc}'(e^{\parsig + i \theta})|^{-\eta} \dx \theta + B_n Z_{n \capc}^* \\
&\leq (1+A_n + B_n) Z_{n\capc}^*.
\end{align*}
Similarly, we can show that
\[
Z_n \geq (1-A_n-B_n) Z_{n\capc}^*. 
\]
Since $A_n+B_n \to 0$ as $\capc \to 0$ there exists $0<c_2 \leq c_1$, depending only on $T$ and $\eta$, such that $A_n + B_n < 1/2$ provided $\capc < c_2$. Assume from now on that $\capc < c_2$.
Hence, if $|\theta-\theta_n| < \capc^{1/2}$ then,  
\[
(1-\alpha_n) h_{n+1}^*(\theta|\theta_{n}) < h_{n+1}(\theta) < (1+\alpha_n) h_{n+1}^*(\theta|\theta_n)
\]
where $\alpha_n = 7(A_n + B_n)$. Equivalently
\[
(1-\alpha_n) h_{n+1}^*(\theta|0) < h_{n+1}(\theta+\theta_n) < (1+\alpha_n) h_{n+1}^*(\theta|0).
\]

As in the proof of Theorem \ref{thm: auxslitconv}, we choose to sample $\theta_k$ from the interval $[\theta_{k-1}-\pi, \theta_{k-1}+\pi)$ and so we can write $\theta_n = u_2+\cdots+u_n$ where the $u_k$ are $[-\pi,\pi)$-valued random variables and, conditional on $\mathcal{F}_{k-1}$, $u_k=\theta_k-\theta_{k-1}$ has distribution function $h_k(\theta+\theta_{k-1})$. We write
\begin{align}
\theta_n = M_n + \sum_{k=1}^n \mathbb{E} \left ( \left . u_k \mathbf{1}_{\{|u_k| \leq k^{\lambda} \parsig (\log \capc^{-1})^{2 \lambda} \}} \right | \mathcal{F}_{k-1} \right ) + \sum_{k=1}^n u_k \mathbf{1}_{\{|u_k| > k^{\lambda} \parsig (\log \capc^{-1})^{2 \lambda} \}},
\label{thetadecompo}
\end{align}
where
\[
M_n = \sum_{k=1}^n \left ( u_k \mathbf{1}_{\{|u_k| \leq k^{\lambda} \parsig (\log \capc^{-1})^{2 \lambda} \}} - \mathbb{E} \left ( \left . u_k \mathbf{1}_{\{|u_k| \leq k^{\lambda} \parsig  (\log \capc^{-1})^{2 \lambda} \}} \right | \mathcal{F}_{k-1} \right ) \right )
\]
is a martingale.

We first show  $M_n$ is small with high probability. By Lemma \ref{slitmapmomentbounds},
\begin{align*}
\mathbb{E} \left ( \left . |u_k|^2 \mathbf{1}_{\{|u_k| \leq k^{\lambda} \parsig (\log \capc^{-1})^{2 \lambda} \}} \right | \mathcal{F}_{k-1} \right ) 
&\leq (1+\alpha_{k-1})\int_{|\theta| \leq k^{\lambda} \parsig (\log \capc^{-1})^{2 \lambda}} |\theta|^2 h_k^*(\theta|0) \dx \theta \\
&\leq 
\begin{cases} 
A \parsig^2 k^{(3-\eta) \lambda} (\log \capc^{-1})^{2\lambda(3-\eta)} &\mbox{ if } 1<\eta<3 \\
A \parsig^2 (\log \capc^{-1})^2 &\mbox{ if } \eta \geq 3, \\
\end{cases} 
\end{align*}
for some constant $A$ depending only on $T$ and $\eta$. Hence $M_n$ is a martingale with quadratic variation
\[
\langle M_{n \wedge N_T} \rangle \leq A n^{2 \lambda} \parsig^2 (\log \capc^{-1})^{4 \lambda}.
\]
By Freedman's version of Bernstein's inequality, see \cite[Proposition 1]{F75},  we obtain that
\[
\mathbb{P}\left ( |M_n| >  \parsig n^{\lambda} (\log \capc^{-1})^{6 \lambda} / 2 \mbox{ for some } n\leq N_T \right ) \leq 2 \exp \left ( - \frac{(\log c^{-1})^{4 \lambda}}{8(A+1)} \right ) \to 0 \quad \textrm{as}\quad \capc\to 0
\]
as desired.

We next turn to the second term in \eqref{thetadecompo}. We use that
\begin{align*}
&\mathbb{E} \left ( \left . u_k \mathbf{1}_{\{|u_k| \leq k^{\lambda} \parsig (\log \capc^{-1})^{2 \lambda} \}} \right | \mathcal{F}_{k-1} \right ) \\
= &\int_{|\theta| \leq k^{\lambda} \parsig (\log \capc^{-1})^{2 \lambda}} \theta h_{k}(\theta + \theta_{k-1}) \dx \theta \\
= &\int_{|\theta| \leq k^{\lambda} \parsig (\log \capc^{-1})^{2 \lambda}} \theta h^*_{k}(\theta|0) \dx \theta + \int_{|\theta| \leq k^{\lambda} \parsig (\log \capc^{-1})^{2 \lambda}} \theta (h_{k}(\theta + \theta_{k-1})-h_k^*(\theta|0) ) \dx \theta \\
= &\int_{|\theta| \leq k^{\lambda} \parsig (\log \capc^{-1})^{2 \lambda}} \theta (h_{k}(\theta + \theta_{k-1})-h_k^*(\theta|0) ) \dx \theta,
\end{align*}
by the symmetry of $h^*_k(\theta|0)$. Hence, again by Lemma \ref{slitmapmomentbounds},
\begin{align*}
\left | \mathbb{E} \left ( \left . u_k \mathbf{1}_{\{|u_k| \leq k^{\lambda} \parsig  (\log \capc^{-1})^{2 \lambda} \}} \right | \mathcal{F}_{k-1} \right ) \right |
&\leq \int_{|\theta| \leq k^{\lambda} \parsig (\log \capc^{-1})^{2 \lambda}} |\theta| |h_{k}(\theta + \theta_{k-1})-h_k^*(\theta|0) | \dx \theta \\
&\leq \alpha_{k-1} \int_{|\theta| \leq k^{\lambda} \parsig (\log \capc^{-1})^{2 \lambda}} |\theta| h_k^*(\theta|0) \dx \theta \\
&\leq 
\begin{cases} 
A \alpha_{k-1} \parsig k^{(2-\eta) \lambda} (\log \capc^{-1})^{2 \lambda(2-\eta)} &\mbox{ if } 1<\eta<2 \\
A \alpha_{k-1} \parsig (\log \capc^{-1})^2 &\mbox{ if } \eta \geq 2, \\
\end{cases} 
\end{align*}
for some constant $A$ depending only on $T$ and $\eta$. Therefore, if $1<\eta<2$,
\begin{align*}
\left | \sum_{k=1}^n \mathbb{E} \left ( \left . u_k \mathbf{1}_{\{|u_k| \leq k^{\lambda} \parsig (\log \capc^{-1})^{2 \lambda} \}} \right | \mathcal{F}_{k-1} \right ) \right | 
\leq &A \parsig (\log \capc^{-1})^{2\lambda(2-\eta)} \sum_{k=1}^n \alpha_{k-1} k^{(2-\eta) \lambda} \\
\leq & \parsig n^{\lambda} (\log \capc^{-1})^{6 \lambda} \left ( A n^{-(\eta-1)\lambda +1} \alpha_n (\log \capc^{-1})^{-2 \lambda(1+\eta)} \right ),
\end{align*}
and if $\eta \geq 2$,
\begin{align*}
\left | \sum_{k=1}^n \mathbb{E} \left ( \left . u_k \mathbf{1}_{\{|u_k| \leq k^{\lambda} \parsig (\log \capc^{-1})^{2 \lambda} \}} \right | \mathcal{F}_{k-1} \right ) \right | 
\leq &A \parsig (\log \capc^{-1})^2 \sum_{k=1}^n \alpha_{k-1}  \\
\leq & \parsig n^{\lambda} (\log \capc^{-1})^{6 \lambda} \left ( A n^{1-\lambda} \alpha_n (\log \capc^{-1})^{-2(3\lambda-1)} \right ).
\end{align*}

By our choice of $\gamma$, there exists $0<c_3 \leq c_2$, depending only on $T$ and $\eta$, such that 
\[
\left | \sum_{k=1}^n \mathbb{E} \left ( \left . u_k \mathbf{1}_{\{|u_k| \leq k^{\lambda} \parsig (\log \capc^{-1})^{2 \lambda} \}} \right | \mathcal{F}_{k-1} \right ) \right | < \parsig n^{\lambda} (\log \capc^{-1})^{6 \lambda}/2
\]
provided $\capc < c_3$. From now on assume that $\capc < c_3$.

Finally, we deal with the last term in \eqref{thetadecompo}. The same computation as used to bound $Z_n$ can be used to show that
\[
\mathbb{P}(|u_k| \geq \capc^{1/2}; \ k \leq N_T) \leq B_k.
\]
We also have
\begin{align*}
\mathbb{P}(k^{\lambda} \parsig (\log \capc^{-1})^{2 \lambda} <|u_k| \leq \capc^{1/2}) &\leq A(1+\alpha_{k-1}) \int_{k^{\lambda} \parsig (\log \capc^{-1})^{2 \lambda}}^{\capc^{1/2}} \frac{1}{\parsig}\left(1+\left(\frac{\theta}{\parsig}\right)^2\right)^{-\eta/2} \dx \theta \\
&\leq A \int_{k^{\lambda} (\log \capc^{-1})^{2 \lambda}}^{\infty} \left(1+\theta^2\right)^{-\eta/2} \dx \theta \\
&\leq A k^{-\lambda(\eta-1)} ( \log \capc^{-1})^{-2 \lambda (\eta-1)}.
\end{align*}
Hence, putting these two bounds together,
\begin{align*}
&\mathbb{P} \left ( \sum_{k=1}^n u_k \mathbf{1}_{\{|u_k| > k^{\lambda} \parsig (\log \capc^{-1})^{2 \lambda} \}} \neq 0 \mbox{ for some } n \leq N_T \right ) \\
\leq &\mathbb{P}(|u_k| > k^{\lambda} \parsig (\log \capc^{-1})^{2 \lambda} \mbox{ for some } k \leq N_T) \\
\leq &A \sum_{k=1}^{N}  \left ( k^{-\lambda(\eta-1)} ( \log \capc^{-1})^{-2 \lambda (\eta-1)} + B_k \right ) \\
\leq &A \left ( ( \log \capc^{-1})^{-1}  + NB_N \right ) \to 0
\end{align*}
since $\parsig < \sigma_0$.

But on the high probability event
\begin{multline*}
\left \{ |M_n| <  \parsig n^{\lambda} (\log \capc^{-1})^{6 \lambda} / 2  \mbox{ for all } n \leq N_T \right \} \\
\cap \left\{\left | \sum_{k=1}^n \mathbb{E} \left ( \left . u_k \mathbf{1}_{\{|u_k| \leq k^{\lambda} \parsig (\log \capc^{-1})^{2 \lambda} \}} \right | \mathcal{F}_{k-1} \right ) \right | < \parsig n^{\lambda} (\log \capc^{-1})^{6 \lambda}/2\right\}\\ \cap
\left \{ \sum_{k=1}^n u_k \mathbf{1}_{\{|u_k| > k^{\lambda} \parsig (\log \capc^{-1})^{2 \lambda} \}} = 0 \mbox{ for all } n \leq N_T \right \}
\end{multline*}
we have
\[
\sup_{n \leq N_T} |\theta_n| < \parsig n^{\lambda} (\log \capc^{-1})^{6 \lambda}
\]
and hence $N_T=N$. 
\end{proof}

\subsection{Modifications of the model}

One criticism that can be levelled at the ALE$(0,\eta)$ model, from the point of view of modelling physical phenomena, is that the conformal mappings distort the sizes of particles as they are added to the growing cluster. Using the result proved above that the scaling limit of the ALE$(0,\eta)$ cluster is a growing slit, it can be shown that the size of the $n$th particle is approximately equal to
$d(\capc n) - d(\capc(n-1))$. Using the expression for $d(t)$ in \eqref{sticklength}, we obtain
\[
d(\capc n) - d(\capc(n-1)) \asymp 
\begin{cases}
\frac{2\capc^{1/2}}{n^{1/2}+(n-1)^{1/2}} &\quad \mbox{if} \quad \capc n \ll 1; \\
2 \capc e^{\capc n} &\quad \mbox{if} \quad \capc n \gg 1.
\end{cases}
\]  
In particular, the first particle is of size approximately $2 \capc^{1/2}$, whereas all subsequent particles are strictly smaller.

A number of modifications to the model are possible which result in clusters where all of the particles are roughly the same size. The simplest modification (cf. \cite{JST15}) is to recursively choose a deterministic sequence of capacities with $c_1=\capc$ and $c_n$ satisfying
\[
d(C_n) - d(C_{n-1}) = d(\capc)\quad
\textrm{where}\quad
C_n = \sum_{j=1}^n c_j.
\]
Another modification (see \cite{Has01, MatJen02}) is to take the logarithmic capacity of the $n$th particle to be 
\[
c_n = \capc |\Phi_{n-1}'(e^{\tilde{\parsig}+i\theta_n})|^{-2}
\]
for some regularization parameter $\tilde{\parsig}>0$, not necessarily equal to the angular regularization parameter $\parsig$. Closely related (see \cite{CM01, RZ05}), is to choose logarithmic capacity $c_n$ corresponding to slit length
\[
d_n = \inf\{d > 0: d |\Phi_{n-1}'((1+d)e^{i \theta_n})| = d(\capc) \}.
\]

In each of these modified models, the total capacity of the cluster no longer grows linearly in the number of particles and is potentially random. It is therefore necessary to modify the timescale in which to obtain scaling results. More precisely, given some fixed $T>0$, let
\[
n(t)= \sup \{ n: C_n < t \} \quad \mbox{for} \quad t \leq T,
\] 
and set $N=n(T)$. The event $\Omega_N$ can then be defined as before.

It is relatively straightforward to verify that the proof and conclusion of Theorem \ref{slitstheorem} still hold for these modified models (and further generalisations). We only state the modified result for $\eta>1$, as the case $\eta<1$ is identical to that for the Markov model, for any choice of logarithmic capacity sequence.

\begin{corollary}
For $\eta>1$ and $\capc>0$, define $\sigma_0$ as in Theorem \ref{slitstheorem} and take $\parsig<\sigma_0$. Consider a sequence of conformal mappings, constructed as in \eqref{aggproc} from sequences $\{\theta_k \}_{k=1}^{\infty}$ and $\{c_k\}_{k=1}^\infty$, where (without loss of generality) $\theta_1=0$ and, conditional on $\mathcal{F}_{n-1} = \sigma(\theta_k, c_k: 1 \leq k \leq n-1)$,  $\theta_{n}$ are given by \eqref{aleangles}. 

Provided there exists some constant $A>0$, depending only on $T$ and $\eta$, such that
\[
\mathbb{P}(c_k \geq A \capc \mbox{ for all } k=1, \dots N) \to 1
\]
as $\capc \to 0$, it holds that $\mathbb{P}(\Omega_N) \to 1$ as $\capc \to 0$. Furthermore, such a constant $A$ exists for the three modifications defined above as well as for $\mathrm{ALE}(\alpha,\eta)$ for any $\alpha>0$.

In this case, for any $r>1$ and $T<\infty$,
\[
\sup_{t \leq T} \sup_{|z|>r}|\Psi_t(z) - f_t(z) | \to 0 \quad \textrm{ in probability as }\quad \capc\to 0,
\]
where $\Psi_t$ is the solution to \eqref{loewnerPDEdrivingfunct} corresponding to the modified model,
and hence the cluster $K_{t}$ converges in the Hausdorff topology to a disk with slit of logarithmic capacity $t$ attached at position $1$.
\end{corollary}
Note that, as we do not impose an upper bound on each logarithmic capacity $c_k$, it is no longer necessarily the case that $n(t)\to t$ as $\capc \to 0$. For his reason, we need to compare $f_t$ with $\Psi_t$, rather than $\Phi_{n(t)}$ as in the previous result.
\begin{proof}
The proof consists of checking step by step that each inequality in the proofs of Lemma \ref{phinbounds} and Theorem \ref{slitstheorem} still holds (possibly with new constants).
 The only changes are that we compare $\Phi_n$ to 
 \[f_{C_n}^{\theta_n}=f_{c_1}^{\theta_n}\circ\cdots \circ f_{c_n}^{\theta_n}\] 
 instead of $f_{\capc n}^{\theta_n}$ and we need to define 
\[
N_T=\inf\left\{k\geq 1\colon |\theta_k|> \parsig k^{\lambda} (\log \capc^{-1})^{6 \lambda} \mbox{ or } c_k < A \capc \right\}\wedge N
\]
and then use the additional assumption in the statement of the corollary to show that $N_T=N$ with high probability.

To show that the additional assumption holds for the modified models defined above, it is enough to show that, so long as $n \leq N_T$, there exists some constant $A$ (depending only on $T$ and $\eta$), such that
\[
|\Phi_{n-1}'(e^{\tilde{\parsig}+i\theta_n})|^{-1} > A.
\]  
But this follows by using the (analogous) estimates in Lemma \ref{phinbounds} for the modified model and observing that there exists some constant $A'$ (depending only on $T$) such that
\[
|f_t'(z)| < A'.
\]  
whenever $|\arg(z)| \leq \beta_t/2$ and $t \leq T$.
\end{proof}

\section{Estimates on conformal maps via Loewner's equation}\label{conformal}
We now obtain refined estimates on the distance between solutions to the Loewner equation in terms of the distance between their driving functions, in the special case when the driving functions are close to constant. These will enable us to prove Lemma 8. Generic estimates between conformal maps tend to blow up close to the boundary (as seen in, for example, Proposition \ref{loewnercontprop}).
As we wish to compare $|\Phi_n'(e^{\parsig+i\theta})|$ to $|(f_{n\capc}^{\theta_n})'(e^{\parsig+i\theta})|$ when $\parsig$ is typically much smaller than the difference between the respective driving functions, we need bespoke estimates which behave well close to the boundary.

Suppose $\Psi^j_t(z)$ is the solution to the Loewner equation \eqref{loewnerPDEdrivingfunct} with driving function $\xi^j$, for $j=0,1$. For fixed $T>0$, let $u_t^j(z)$ be the corresponding reverse-time Loewner flows defined in \eqref{loewnerODE}, so that $\Psi^j_T(z) = u^j_T(z)$ and $(\Psi^j_T)'(z) = (u^j_T)'(z)$. In Section \ref{revanal}, we compare $\Psi^1_T(z)$ to $\Psi^0_T(z_0)$ under the assumption that $\Psi^0_T(z_0)$ (or, more precisely, $u^0_t(z_0)$, for $0 \le t \leq T$) is ``known''. Specifically, we find conditions on $\| \xi^1-\xi^0 \|_T = \sup_{t \leq T}  |e^{i\xi^1_t}-e^{i\xi^0_t}|$ and $|z-z_0|$, which depend on $u^0_t(z_0)$ and $(u^0_t)'(z_0)$, under which $|u^1_t(z)-u^0_t(z_0)|$ can be shown to be small. 

In Section \ref{slitcomp} we interpret this result when $\xi^0 \equiv 0$. This enables us to compare $\Psi'_T(z)$ to $f'_T(z)$ when $\xi$, the driving function of $\Psi$, is close to zero. Specifically, we obtain refined estimates in the case when $\arg z$ is close to 0 and in the case when $|z|$ is close to 1. We also obtain cruder estimates which apply in the intermediate regime between these two cases which are used in the proof of Lemma \ref{phinbounds} to ``glue'' the two results together.

\subsection{Analysis of the reverse-time Loewner flow}
\label{revanal}

Define $h: \Delta \times \TT \to \CC$ by
\[
h(u,v) = u \frac{u v +1}{u v - 1}
\]
so, by \eqref{loewnerODE},
\[
\partial_t u_t^j(z) = h(u_t^j(z), e^{-i\xi^j_{T-t}}), \quad j=0,1.
\]
Observe that
\begin{align}
\label{hparder}
\frac{\partial h}{\partial u}(u,v) &= 1 - \frac{2}{(u v - 1)^2}, \nonumber \\
\frac{\partial h}{\partial v}(u,v) &= -\frac{ 2u^2}{(u v - 1)^2}.
\end{align}
Since
\[
\partial_t (u_t^j)'(z) = \frac{\partial h}{\partial u}(u_t^j(z), e^{-i\xi^j_{T-t}}) (u_t^j)'(z),
\]
using $(u_0^j)'(z)=1$, we therefore obtain
\begin{equation}
\label{uderdef}
(u_t^j)'(z) = \exp \left ( t - \int_0^t \frac{2 \dx s}{(u^j_s(z) e^{-i \xi^j_{T-s}} - 1)^2}\right ).
\end{equation}
It is also convenient to write $u_t^{j}(z)=r_t^{j}(z)e^{i \vartheta^{j}_t(z)}$ where $r_t^{j}(z) \geq 1$ and $\vartheta^{j}_t(z) \in \RR$ with $\vartheta^{j}_0(z) \in (-\pi, \pi]$. Substituting this into \eqref{loewnerODE} and separating $\mathrm{Re}[(u^{j}_t(z)e^{-i\xi^j_{T-t}}+1)/(u^{j}_t(z)e^{-i\xi^j_{T-t}}-1)]$ and $\mathrm{Im}[(u^{j}_t(z)e^{-i\xi^j_{T-t}}+1)/(u^{j}_t(z)e^{-i\xi^j_{T-t}}-1)]$
we obtain the two differential equations
\begin{equation}
\partial_tr^{j}_t=r^{j}_t\frac{(r^{j}_t)^2-1}{(r^{j}_t)^2-2r^{j}_t\cos (\vartheta^{j}_t-\xi^j_{T-t})+1}
\label{radiusreverse}
\end{equation}
and
\begin{equation}
\partial_t\vartheta^{j}_t=-2\frac{r^{j}_t\sin (\vartheta^{j}_t-\xi^j_{T-t})}{(r^{j}_t)^2-2r^{j}_t\cos (\vartheta^{j}_t-\xi^j_{T-t})+1}
\label{anglereverse}
\end{equation}
(where we have suppressed the dependence on $z$ to ease notation).

We observe that the right hand side of \eqref{radiusreverse} is non-negative and maximised when $\vartheta^{j}_t-\xi^j_{T-t}=0$. 
In this case, the differential equation
\[\partial_tr^{j}_t=r^{j}_t\frac{r^{j}_t+1}{r^{j}_t-1}\]
can be solved explicitly, 
\[r^{j}_t(z)=\frac{1}{2|z|}\left(e^t |z|^2 + 2 e^t |z| +e^t- e^{t/2}(|z| + 1)\sqrt{e^t (|z| + 1)^2- 4|z|} - 2 |z|\right).\]
Noting that
\[r_t^{j}(z)\leq e^t\frac{(|z|+1)^2}{|z|}\leq 4e^t|z|, \quad |z|>1,\]
we obtain the crude estimate
\begin{equation}
\label{crudeu}
|z| \leq |u_t^j(z)| \leq 4|z| e^t.
\end{equation}

\begin{lemma}
\label{ucomp} 
Suppose $z_0 \in {\Delta}$, $T>0$ and $\xi^0:(0,T] \to \RR$ are given and let
\[
\Lambda_t = \int_0^t \frac{2 |u_s^0(z_0)|^2 \dx s}{|(u_s^0)'(z_0)||u_s^0(z_0)e^{i \xi^0_{T-s}}-1|^2}.
\]
There exists some absolute constant $A$ such that, for all $|z| > 1$  satisfying 
\begin{align}
\label{zassump}
|z-z_0| &\leq A^{-1} \inf_{0 \leq t \leq T} \left ( \frac{|u_t^0(z_0)e^{-i \xi^0_{T-t}} -1|}{|(u_t^0)'(z_0)|} \wedge \left ( \int_0^t \frac{|(u_s^0)'(z_0)| }{ |u_s^0(z_0)e^{-i \xi^0_{T-s}} -1|^3} \dx s \right )^{-1} \right ), 
\end{align}
we have, for all $0 \leq t \leq T$,
\[
\left | \log \frac{u^0_t(z) - u^0_t(z_0)}{(z-z_0)(u^0_t)'(z_0)}\right | \leq A |z-z_0| \int_0^t \frac{|(u^0_s)'(z_0)|\dx s}{|u_s^0(z_0)e^{-i \xi^0_{T-s}} -1|^3}  
\]
(where we interpret the left hand side as being equal to $0$ if $z=z_0$) and
\[
\left | \log \frac{(u^0_t)'(z)}{(u^0_t)'(z_0)} \right | \leq A |z-z_0| \int_0^t \frac{|(u^0_s)'(z_0)|\dx s}{|u_s^0(z_0)e^{-i \xi^0_{T-s}} -1|^3} . 
\]
Furthermore, $A$ can be chosen so that if, in addition, $\xi^1:(0,T] \to \RR$ satisfies
\begin{align}
\label{xiassump}
\|\xi^1-\xi^0\|_T &\leq A^{-1} \inf_{0 \leq t \leq T} \left ( \frac{|u_t^0(z_0)e^{-i \xi^0_{T-t}} -1|}{|(u_t^0)'(z_0)|\Lambda_t+|u^0_t(z_0)|} \wedge \left ( \int_0^t \frac{\Lambda_s|(u_s^0)'(z_0)| + |u^0_s(z_0)|}{ |u_s^0(z_0)e^{-i \xi^0_{T-s}} -1|^3} \dx s \right )^{-1} \right ),
\end{align}
then, for all $0 \leq t \leq T$,
\begin{equation}
\left | u^1_t(z)-u^0_t(z) \right | \leq A|(u^0_t)'(z_0)| \|\xi^1 - \xi^0\|_T \Lambda_t  
\label{uLambdabound}
\end{equation}
and
\[
\left | \log \frac{(u^1_t)'(z)}{(u^0_t)'(z)} \right | \leq A \| \xi^1-\xi^0\|_T \int_0^t \frac{\Lambda_s|(u_s^0)'(z_0)| + |u^0_s(z_0)|}{ |u_s^0(z_0)e^{-i \xi^0_{T-s}} -1|^3} \dx s. 
\]
\end{lemma}
Lemma \ref{ucomp} can be interpreted as telling us that, provided $u^0_t(z_0)$ stays away from $e^{i \xi^0_{T-t}}$, $u_t^1(z)$ will be close to $u_t^0(z_0)$ for sufficiently small $|z-z_0|$ and $\|\xi^1-\xi^0\|_T$. The conditions in \eqref{zassump} and \eqref{xiassump} quantify precisely what is meant by `sufficiently small'.
\begin{remark} 
\begin{itemize}
\item[1.] At first glance, Lemma \ref{ucomp} may not appear to be very illuminating. However, the key point is that all of the bounds have been expressed purely in terms of $u^0_t(z_0)$ for $0 \leq t \leq T$, which enables us to obtain good estimates in situations where we have good control over $u^0_t(z_0)$. The benefit of this approach is demonstrated in Section \ref{slitcomp}. There, $u^0_t(z_0)$ is taken to be the solution corresponding to a constant driver and so the relevant terms may be computed explicitly to yield simple expressions. 
\item[2.] The conditions \eqref{zassump} and \eqref{xiassump} can be simplified by observing that by \eqref{radiusreverse}, for any $g: [0,T] \to [0,\infty)$, 
\[
\int_0^t \frac{g(s) \dx s}{|u_s^0(z_0)e^{-i \xi^0_{T-s}} -1|^2} \leq \sup_{0 \leq s \leq t} g(s) \int_0^t \frac{\partial_s r^0_s}{r^0_s((r^0_s)^2 - 1)} \dx s= \frac{1}{2} \sup_{0 \leq s \leq t} g(s) \log \frac{(|u^0_t(z_0)|^2-1)|z_0|^2}{|u^0_t(z_0)|^2 (|z_0|^2-1)}.
\]
Therefore 
\[
\inf_{0 \leq t \leq T} \left ( g(t)^{-1} \wedge \left ( \int_0^t \frac{g(s) }{ |u_s^0(z_0)e^{-i \xi^0_{T-s}} -1|^2} \dx s \right )^{-1} \right )
\]
can be replaced by
\[
\inf_{0 \leq t \leq T} g(t)^{-1} \left ( \frac{1}{2}\log \frac{|z_0|}{ |z_0|-1} \right )^{-1}.
\]
However, in the cases we are interested in, it is possible to eliminate the $\log$ term by computing the integral explicitly.
\item[3.] 
Although this result is most powerful when applied to specific choices of $z_0$ and $\xi^0$, it can be used to provide generic estimates too.

Observe that, by \eqref{uderdef} and the crude estimates on $|u^0_t(z_0)|$ in \eqref{crudeu},
\begin{align*}
\Lambda_t &\leq \int_0^t \exp \left ( -s + \int_0^s \frac{2 |u_s^0(z_0)| \dx r}{|u^0_r(z_0) e^{-i \xi^0_{T-r}} - 1|^2}\right ) \frac{8 |z_0|e^s |u_s^0(z_0)|}{|u^0_s(z_0) e^{-i \xi^0_{T-s}} - 1|^2} \dx s \\
&=  4|z_0|\left ( \exp \left ( \int_0^t \frac{2 |u^0_s(z_0)| \dx s}{|u^0_s(z_0) e^{-i \xi^0_{T-s}} - 1|^2} \right ) - 1 \right ) 
\end{align*}
and, by \eqref{radiusreverse},
\[
\int_0^t \frac{2 |u^0_s(z_0)| \dx t}{|u_s^0(z_0)e^{-i \xi^0_{T-s}} -1|^2} = \int_0^t \frac{2 \partial_sr^{0}_s \dx s}{(r^{0}_s)^2-1} = \log \frac{(|u^0_t(z_0)|-1)(|z_0|+1)}{(|u^0_t(z_0)|+1)(|z_0|-1)}.
\]
Hence it follows from \eqref{uLambdabound} (taking $z=z_0$) that there exists some absolute constant $A$ such that
\[
|\Psi^1_T(z) - \Psi^0_T(z)| \leq A  \|\xi^1 - \xi^0\|_T |(\Psi^0_T)'(z)| |z| \frac{(|\Psi^0_T(z)|-1)(|z|+1)}{(|\Psi^0_T(z)|+1)(|z|-1)}.
\]
By using standard distortion estimates to bound $|(\Psi^0_T)'(z)|$, there exists some (possibly different) absolute constant $A$ such that
\[
|\Psi^1_T(z) - \Psi^0_T(z)| \leq \frac{A e^T |z|\|\xi^1 - \xi^0\|_T}{(|z|-1)^2}
\]
(cf Proposition \ref{loewnercontprop}). 

Here we have used only generic information about the two flows. We note that this last estimate is not optimal, however, as we have taken worst-case bounds for both $\Lambda_T$ and $|(\Psi^0_T)'(z)|$, whereas typically these two quantities are bad in different regions. Indeed, one expects the exponent $1$ in the denominator as has been proved in the chordal setting. In fact, one can start from the setting of Proposition \ref{loewnercontprop} to obtain an exponent $1+\delta$ for $\delta > 0$ arbitrarily small (see \cite{F1}). Alternatively, one can localise and use the half-plane case (see \cite{JRW14}). By following the latter approach near the tip of a slit map, one can obtain an estimate that also exploits information about the derivative but with a sub-power correction that we do not get here. 

We emphasize that the case in which we apply this result is not the generic one. We have much information about $|(u_t^0)'(z_0)|$ and the form of the estimates here allows us to use this information efficiently.
\end{itemize}
\end{remark}

\begin{proof}
Set $\delta^j_t=u^j_t(z) - u^0_t(z_0)$ for $j=0,1$. Then $\delta^j_t$ satisfies the ODE
\[
\frac{\dx \delta^j_t}{\dx t} = h(u^j_t(z), e^{-i \xi^j_{T-t}}) - h(u^0_t(z_0), e^{-i \xi^0_{T-t}}).
\]
We shall obtain the desired estimates by linearising this ODE and showing that, under assumptions \eqref{zassump} and \eqref{xiassump}, the higher order terms can be controlled.

Write
\[
\frac{\dx \delta^j_t}{\dx t} 
= \delta^j_t \frac{\partial h}{\partial u}(u^0_t(z_0), e^{-i \xi^0_{T-t}}) + (e^{-i \xi^j_{T-t}}-e^{-i \xi^0_{T-t}}) \frac{\partial h}{\partial v}(u^0_t(z_0), e^{-i \xi^0_{T-t}}) + H^j(t)
\]
where, by direct computation, 
\begin{align*}
H^j(t) = -\frac{2\left ( (\delta^j_t)^2 e^{-i \xi^j_{T-t}} + 2 \delta^j_t (e^{-i \xi^j_{T-t}}-e^{-i \xi^0_{T-t}}) u^0_t(z_0)  + (e^{-i \xi^j_{T-t}}-e^{-i \xi^0_{T-t}})^2(u^0_t(z_0))^2u^j_t(z) \right )}{(u_t^0(z_0)e^{-i \xi^0_{T-t}} -1)^2(u_t^j(z)e^{-i \xi^j_{T-t}} -1)} . 
\end{align*}

Taking $j=0$, we have
\[
\frac{\dx}{\dx t} \log \delta^0_t = \frac{\partial h}{\partial u}(u^0_t(z_0), e^{-i \xi^0_{T-t}}) + (\delta^0_t)^{-1}H^0(t)
\]
and hence, using \eqref{uderdef} and that $(u^0_0)'(z_0)=1$,
\begin{equation}
\label{delta0sol}
\log \frac{\delta^0_t}{(z-z_0)(u^0_t)'(z_0)} = \int_0^t (\delta^0_s)^{-1}H^0(s) \dx s.
\end{equation}
Taking $j=1$, we have
\begin{align*}
\frac{\dx}{\dx t} \left ( \frac{\delta^1_t}{(u^0_t)'(z_0)}\right ) 
&=  \frac{1}{(u^0_t)'(z_0)} \frac{\dx \delta^1_t}{\dx t}  - \frac{\delta^1_t}{(u^0_t)'(z_0)^2} \frac{\dx }{\dx t}\left ( (u^0_t)'(z_0) \right )  \\
&= \frac{1}{(u^0_t)'(z_0)} \left ( (e^{-i \xi^1_{T-t}}-e^{-i \xi^0_{T-t}}) \frac{\dx h}{\dx v} (u^0_t(z_0), e^{-i \xi^0_{T-t}})  +H^1(t) \right )
\end{align*}
and hence, using \eqref{hparder},
\begin{equation}
\label{deltasol}
\frac{\delta^1_t}{(u^0_t)'(z_0)} - (z-z_0) = - \int_0^t \frac{2 (e^{-i \xi^1_{T-s}}-e^{-i \xi^0_{T-s}}) u^0_s(z_0)^2}{(u^0_s)'(z_0)(u_s^0(z_0)e^{i \xi^0_{T-s}}-1)^2} \dx s+ \int_0^t \frac{H^1(s)}{(u^0_s)'(z_0)} \dx s.
\end{equation}
Since
\[
\left | \int_0^t \frac{2 (e^{-i \xi^1_{T-s}}-e^{-i \xi^0_{T-s}}) u^0_s(z_0)^2}{(u^0_s)'(z_0)(u_s^0(z_0)e^{i \xi^0_{T-s}}-1)^2} \dx s \right | \leq \| \xi^1-\xi^0\|_T \Lambda_t
\]
it follows immediately that
\[
|\delta^1_t - (z-z_0)(u^0_t)'(z_0)| \leq |(u^0_t)'(z_0)| \left ( \| \xi^1-\xi^0\|_T \Lambda_t + \int_0^t \frac{|H^1(s)|}{|(u^0_s)'(z_0)|} \dx s\right ).
\]

We next obtain bounds on $H^j(t)$, under the assumption that $t \leq T^j$, where
\begin{align*}
T^j= &\inf \left \{t>0: |\delta^j_t| > 2 |(u^0_t)'(z_0)| \left ( \| \xi^j-\xi^0\|_T \Lambda_t + |z-z_0| \right ) \right \} \wedge T.
\end{align*}
In what follows, we shall show that if we take $A=25$ in assumption \eqref{zassump} then $T^0=T$ and if we take it in \eqref{zassump} and \eqref{xiassump} then $T^1=T$. (Note that we have made no attempt to optimise the value of $A$.)

Using \eqref{xiassump} and \eqref{uderdef}, 
\[
\|\xi^1-\xi^0\|_T \leq \frac{|u_t^0(z_0)e^{-i \xi^0_{T-t}} -1|}{25|u_t^0(z_0)|}
\]
and
\[
|\delta^j_t| \leq \frac{4|u_t^0(z_0)e^{-i \xi^0_{T-t}} -1|}{25}
\]
for all $t \leq T^j$. Hence, 
\begin{align*}
\left | |u_t^j(z)e^{-i \xi^j_{T-t}} - 1| - |u_t^0(z_0)e^{-i \xi^0_{T-t}} -1| \right | 
&\leq |u_t^j(z)e^{-i \xi^j_{T-t}} - u_t^0(z_0)e^{-i \xi^0_{T-t}} | \\
&\leq |u_t^j(z) - u_t^0(z_0)| |e^{-i \xi^j_{T-t}}| + |e^{-i \xi^j_{T-t}} - e^{-i \xi^0_{T-t}}| |u_t^0(z_0)| \\
&\leq |\delta^j_t| + \|\xi^j-\xi^0\|_T |u_t^0(z_0)| \\
&\leq \frac{1}{5}|u_t^0(z_0)e^{-i \xi^0_{T-t}} -1| 
\end{align*}
and so
\[
|u_t^j(z)e^{-i \xi^j_{T-t}} - 1| \geq \frac{4}{5}|u_t^0(z_0)e^{-i \xi^0_{T-t}} -1|.
\]
Also
\[
|u_t^j(z)| \leq |u_t^0(z_0)| + |\delta^j_t| \leq \frac{33}{25} |u_t^0(z_0)|. 
\]
Hence, using the bounds above,
\[
\left |(\delta^0_t)^{-1} H^0(t) \right | \leq \frac{5 |\delta^0_t|}{2|u_t^0(z_0)e^{-i \xi^0_{T-t}} -1|^3} \leq 5 \frac{|z-z_0||(u_t^0)'(z_0)|}{|u_t^0(z_0)e^{-i \xi^0_{T-t}} -1|^3} 
\]
and so, by \eqref{delta0sol}, $T^0=T$ and the first statement in the lemma follows. 
Similarly
\begin{align*}
|H^1(t)| &\leq \frac{5}{2|u_t^0(z_0)e^{-i \xi^0_{T-t}} -1|^3} \left ( (\delta^1_t)^2 + 2 |u_t^0(z_0)| |\delta^1_t| \|\xi^1-\xi^0\|_T + \frac{33}{25} |u_t^0(z_0)|^3 \|\xi^1-\xi^0\|_T^2\right ) \\
&\leq \frac{20 \left ( \| \xi^1-\xi^0\|_T^2 \Lambda_t^2 + |z-z_0|^2 \right )|(u_t^0)'(z_0)|^2}{|u_t^0(z_0)e^{-i \xi^0_{T-t}} -1|^3} + \frac{233|u_t^0(z_0)|^2 \|\xi^1-\xi^0\|_T}{250|u_t^0(z_0)e^{-i \xi^0_{T-t}} -1|^2}.
\end{align*}
By \eqref{xiassump}, we have
\[
\int_0^t \frac{|H^1(s)|}{|(u_s^0)'(z_0)|} \dx s \leq \| \xi^1-\xi^0\|_T \Lambda_t + 20|z-z_0|^2 \int_0^t \frac{|(u_s^0)'(z_0)| }{ |u_s^0(z_0)e^{-i \xi^0_{T-s}} -1|^3} \dx s.
\]
It follows that $T^1=T$ and hence
\[
|\delta^1_t - (z-z_0)(u^0_t)'(z_0)| \leq |(u^0_t)'(z_0)| \left ( 2\| \xi^1-\xi^0\|_T \Lambda_t + 20|z-z_0|^2 \int_0^t \frac{|(u_s^0)'(z_0)| }{ |u_s^0(z_0)e^{-i \xi^0_{T-s}} -1|^3} \dx s \right ).
\]

To obtain estimates on the derivative, we use that
\begin{align*}
\log (u_t^j)'(z) &= \int_0^t \frac{\partial h}{\partial u}(u_s^j(z), e^{-i\xi^j_{T-s}}) \dx s \\
&= \int_0^t \left ( \frac{\partial h}{\partial u}(u_s^0(z_0), e^{-i\xi^0_{T-s}}) + H^j_1(s) \right )\dx s \\
&= \log (u_t^0)'(z_0) + \int_0^t H^j_1(s) \dx s
\end{align*}
where
\begin{align*}
H^j_1(t) = \frac{-2 \left (\delta^j_t e^{-i \xi^j_{T-t}} + (e^{-i \xi^j_{T-t}}-e^{-i \xi^0_{T-t}})u_t^0(z_0) \right )}{(u_t^0(z_0)e^{-i \xi^0_{T-t}} -1)(u_t^j(z)e^{-i \xi^j_{T-t}} -1)} \left ( \frac{1}{u_t^0(z_0)e^{-i \xi^0_{T-t}} -1} + \frac{1}{u_t^j(z)e^{-i \xi^j_{T-t}} -1} \right ). 
\end{align*}
As above,
\begin{align*}
|H^j_1(t)| \leq \frac{25}{2} \left ( \| \xi^j-\xi^0\|_T \frac{|(u^0_t)'(z_0)| \Lambda_t + |u^0_t(z_0)|}{|u_t^0(z_0)e^{-i \xi^0_{T-t}} -1|^3} +|z-z_0| \frac{|(u^0_t)'(z_0)|}{|u_t^0(z_0)e^{-i \xi^0_{T-t}} -1|^3}\right )
\end{align*}
and hence
\[
\left | \log \frac{(u_t^j)'(z)}{(u_t^0)'(z_0)} \right | \leq \frac{25}{2} \left ( \| \xi^j-\xi^0\|_T \int_0^t \frac{|(u^0_s)'(z_0)| \Lambda_s + |u^0_s(z_0)|}{|u_s^0(z_0)e^{-i \xi^0_{T-s}} -1|^3} \dx s +|z-z_0| \int_0^t \frac{|(u^0_s)'(z_0)|}{|u_s^0(z_0)e^{-i \xi^0_{T-s}} -1|^3} \dx s \right ). 
\]
Finally, we observe that, by the same arguments as above, under assumption \eqref{zassump} with $A=25$,
$|u^0_t(z)|/|u^0_t(z_0)|$, $|u^0_t (z) e^{-i \xi^0_{T-t}}-1|/|u^0_t (z_0) e^{-i \xi^0_{T-t}}-1|$ and $|(u^0_t)'(z)|/|(u_t^0)'(z_0)|$ can be bounded above and below by strictly positive absolute constants and hence there exists some absolute constant $A_1 \geq 1$ such that 
\[
|(u_t^0)'(z)| \tilde{\Lambda}_t+|u^0_t(z)|\leq A_1 \left (|(u_t^0)'(z_0)|\Lambda_t+|u^0_t(z_0)|\right)
\]
for all $0\leq t \leq T$, where
\[
\tilde{\Lambda}_t = \int_0^t \frac{2 |u_s^0(z)|^2 \dx s}{|(u_s^0)'(z)||u_s^0(z)e^{i \xi^0_{T-s}}-1|^2}.
\]
Hence, if assumption \eqref{xiassump} holds with $A=25(5/4)^3A_1$, then
\[
\|\xi^1-\xi^0\|_T \leq 25^{-1} \inf_{0 \leq t \leq T} \left ( \frac{|u_t^0(z)e^{-i \xi^0_{T-t}} -1|}{|(u_t^0)'(z)|\tilde{\Lambda}_t+|u^0_t(z)|} \wedge \left ( \int_0^t \frac{\tilde{\Lambda}_s|(u_s^0)'(z)| + |u^0_s(z)|}{ |u_s^0(z)e^{-i \xi^0_{T-s}} -1|^3} \dx s \right )^{-1} \right ),
\]
and so we may set $z=z_0$ in the computation above to get that
\[
|u^1_t(z) - u^0_t(z)| \leq 2|(u^0_t)'(z)| \| \xi^1-\xi^0\|_T \tilde{\Lambda}_T \leq A |(u^0_t)'(z_0)| \| \xi^1-\xi^0\|_T \Lambda_T
\]
and
\[
\left | \log \frac{(u^1_t)'(z)}{(u^0_t)'(z)} \right | \leq A \| \xi^1-\xi^0\|_T \int_0^t \frac{\Lambda_s|(u_s^0)'(z_0)| + |u^0_s(z_0)|}{ |u_s^0(z_0)e^{-i \xi^0_{T-s}} -1|^3} \dx s, 
\]
as required.
\end{proof}

\subsection{Small driving functions}
\label{slitcomp}

In this section, we explicitly evaluate $u^0_t(z_0)$ and $(u^0_t)'(z_0)$ when $\xi^0 \equiv 0$ and either $\arg z_0 = 0$ or $|z_0|=1$. This enables us to compare $\Psi_T(z)$ to the slit map $f_T(z)$ when $\xi$, the driving function of $\Psi$, is close to zero. Since $\xi^0_{T-t}$ does not depend on $T$, $u^0_t(z_0)=f_t(z_0)$ and $(u^0_t)'(z_0)=f'_t(z_0)$ for all $t \geq 0$. We could therefore, in principle, just substitute the estimates from Section \ref{slitder} into Lemma \ref{ucomp}. However, instead we observe that in these two cases solving the pair of differential equations \eqref{radiusreverse} and \eqref{anglereverse} reduces to solving a single ordinary differential equation, and we are able to obtain explicit solutions directly. 

First suppose that $z_0=r > 1$. Set $u^0_t(z_0)=r^0_te^{i \vartheta^0_t}$. From \eqref{anglereverse} it is immediate that $\vartheta^0_t=0$ for all $t > 0$. Substituting this into
\eqref{radiusreverse} we get
\begin{equation}
\label{drdef}
\partial_tr^{0}_t=r^{0}_t\frac{r^{0}_t+1}{r^{0}_t-1}.
\end{equation}
Solving this gives
\[
\log \left ( \frac{(r_t^0+1)^2 r}{r^0_t (r+1)^2} \right ) = t
\]
or
\begin{equation}
\label{rdef}
u^0_t(z_0) = r^0_t = \frac{(r+1)^2 e^t}{2r} \left ( 1 + \sqrt{1 - \frac{4re^{-t}}{(r+1)^2}} \right ) - 1.
\end{equation}
Observe that if $r=1$, then $r^0_t = d(t) + 1$.

Now suppose $z_0=e^{i \theta}$ where $|\theta| \in (0, \pi)$. Although $u^0_t(z_0)$ is not explicitly defined when $|z_0|=1$, $u^0_t(z)$ for $|z|>1$ can be continuously extended to the boundary of the unit disk in a well-defined way, so this is the interpretation we put on $u^0_t(e^{i \theta})$. 

From \eqref{radiusreverse} it is immediate that $r^0_t=1$ for all $t \leq \inf\{ t > 0: u^0_t(e^{i \theta})=1\}$. Substituting this into
\eqref{anglereverse} we get
\[
\partial_t\vartheta^0_t=-\frac{\sin \vartheta^0_t}{1-\cos \vartheta^0_t}=-\cot \frac{\vartheta^0_t}{2}.
\]
Solving this gives
\begin{equation}
\label{varthetadef}
\vartheta^0_t = \vartheta^0_t(e^{i \theta}) = \cos^{-1}\left( (1+\cos \theta)e^t -1 \right)
\end{equation}
and hence 
\[
\inf\{ t > 0: u^0_t(e^{i \theta})=1\} = \log \frac{2}{1 + \cos \theta}.
\]

\begin{corollary}
\label{dertipetc}
Suppose $\Psi_t(z)$ is the solution to the Loewner equation \eqref{loewnerPDEdrivingfunct}.
\begin{itemize}
\item[(i)] (Near the tip). There exists some absolute constant $A$ such that, for all $|z|>1$ and $T>0$ satisfying $\|\xi\|_T + |\arg z| \leq A^{-1} (|z|-1) / |z|$,
we have 
\[
\left | \log \frac{|\Psi'_T(z)|}{|f_T'(z)|} \right | \leq \left (  \frac{ A|z|(\|\xi\|_T + |\arg z|)}{|z|-1} \right )^2.
\]
\item[(ii)] (Away from the tip).
There exists some absolute constant $A$ such that, for all $|z|>1$ and $T>0$ satisfying
\[
T \leq \log \frac{2}{1+\cos(\arg z)}
\]
and
\[
\|\xi\|_T + |z|-1 \leq A^{-1} e^{-T/2} \cot \frac{\arg z }{2} \tan \frac{\vartheta^0_T}{2}\sqrt{1-\cos \vartheta^0_T},
\]
where $\vartheta^0_t$ is defined as in \eqref{varthetadef} with $\theta=\arg z$, we have 
\[
\left | \log \frac{|\Psi'_T(z)|}{\tan \frac{\arg z}{2} \cot \frac{\vartheta^0_T}{2}} \right | \leq \frac{A\sqrt{e^T(e^T-1)}(\|\xi\|_T+|z|-1)}{1-\cos \vartheta^0_T} \leq 1,
\]
\[
1 - \cos (\arg \Psi_T(z)) \leq A (1-\cos \vartheta^0_T),
\]
and
\[
|\Psi_T(z)| - 1 \geq A^{-1}(|z|-1)\tan \frac{\arg z}{2} \cot \frac{\vartheta^0_T}{2}.
\]
\end{itemize}
\end{corollary}
\begin{proof}
\begin{itemize}
\item[(i)] 
Set $z_0=|z|$ and define $r^0_t$ as in \eqref{rdef}, with $r=|z|$. Using \eqref{drdef}, we compute $|(u^0_t)'(z_0)|$ and $\Lambda_t$ from Lemma \ref{ucomp}. By \eqref{uderdef},
\begin{align*}
|(u^0_t)'(z_0)| &= e^t \exp \left ( - \int_0^t \frac{2 \dx s}{(r^0_s - 1)^2}\right ) \\
&= e^t \exp \left ( - \int_0^t \frac{2 \partial_sr^{0}_s }{r^0_s((r^0_s)^2 - 1)} \dx s\right ) \\
&= e^t \frac{(r^2-1)(r^0_t)^2}{((r^0_t)^2 - 1)r^2} \\
&= \frac{(r-1)r^0_t(r^0_t+1)}{(r^0_t - 1)r(r+1)} \\
&\leq e^t.
\end{align*}
Therefore, again using \eqref{drdef},
\[
\Lambda_t = \int_0^t \frac{(r^0_s - 1)r(r+1)}{(r-1)r^0_s(r^0_s+1)} \frac{2 (r^0_s)^2}{(r^0_s-1)^2} \dx s = \frac{2r(r_t^0-r)}{(r-1)(r_t^0+1)}
\]
and so
\[
|(u^0_t)'(z_0)| \Lambda_t = \frac{2r^0_t (r_t^0-r)}{(r^0_t - 1)(r+1)} \leq r^0_t.
\]
Hence
\[
\frac{|(u^0_t)'(z_0)| \Lambda_t + |u^0_t(z_0)|}{|u_t^0(z_0) -1|} \leq \frac{2r^0_t}{r^0_t-1} \leq \frac{2|z|}{|z|-1}
\]
and
\[
\int_0^t\frac{|(u^0_s)'(z_0)| \Lambda_s + |u^0_s(z_0)|}{|u_s^0(z_0) -1|^3} \dx s
\leq \int_0^t \frac{2r^0_s}{(r^0_s-1)^3} \dx s \leq \frac{1}{|z|-1}.
\]
Here we have used that $r_t^0 \geq |z|$ for all $0 \leq t \leq T$ in each of the final inequalities in the preceding two displays.
Similarly
\[
\frac{|(u^0_t)'(z_0)|}{|u_t^0(z_0) -1|} \leq \frac{1}{|z|-1}
\]
and
\[
\int_0^t \frac{|(u^0_s)'(z_0)|}{|u_s^0(z_0) -1|^3} \dx s \leq \frac{1}{2|z|(|z|^2-1)}.
\]
By Lemma \ref{ucomp}, using that $r^0_t \leq 4|z|e^t$, we get
\[
\left | \log \frac{\Psi'_T(z)}{f_T'(z)} \right | \leq \frac{A \|\xi\|_T}{|z|-1}.
\]
By using that $u^0_t(z_0)$ and hence $(u^0_t)'(z_0)$ are purely real, that
\[
\left | \RRe (e^{i \xi_t}) - 1 \right | \leq \| \xi \|_T^2
\]
and that
\[
\left | \RRe z - |z| \right | \leq |z| (\arg z)^2,
\]
it is possible to repeat the computations in the proof of Lemma \ref{ucomp} for the real parts of $u^1_t(z)$ and $\log (u^1_t)'(z)$ to obtain the stronger bound
\[
\left | \log \frac{|\Psi'_T(z)|}{|f_T'(z)|} \right | = \left | \RRe \log \frac{\Psi'_T(z)}{f_T'(z)} \right | \leq \left (  \frac{ A|z|(\|\xi\|_T + |\arg z|)}{|z|-1} \right )^2.
\]
We omit the details as the argument is almost identical to that used in the proof of Lemma \ref{ucomp}.
\item[(ii)] 
Set $z_0 = e^{i \arg z}$. If $0 \leq t \leq T < \log \frac{2}{1 + \cos (\arg z)}$, then defining $\vartheta^0_t$ as in \eqref{varthetadef}, with $\theta=\arg z$,
\[
|u_t^0(z_0) - 1|^2 = 2(1-\cos \vartheta^0_t),
\]
\begin{align*}
|(u^0_t)'(z_0)| &= \exp \left ( \int_0^t \frac{\dx s}{1-\cos \vartheta^0_s}\right ) \\
&= \exp \left ( -\int_0^t \frac{ \partial_s\vartheta^0_s }{\sin \vartheta^0_s} \dx s \right ) \\
&= \tan \frac{\theta}{2} \cot \frac{\vartheta^0_t}{2},
\end{align*}
and
\[
\Lambda_t = \cot \frac{\theta}{2} \int_0^t \frac{\tan \frac{\vartheta^0_t}{2}}{1-\cos \vartheta^0_s} \dx s = 1-\cot \frac{\theta}{2} \tan \frac{\vartheta^0_t}{2}.
\]
By standard trigonometric identities, and using the explicit value of $\vartheta^0_t$ from \eqref{varthetadef},
\begin{equation}
\label{tancotid}
\tan \frac{\theta}{2} \cot \frac{\vartheta^0_t}{2} = \sqrt{\frac{(1-\cos \theta)(1+\cos \vartheta^0_t)}{(1+\cos \theta)(1-\cos \vartheta^0_t)}} = \sqrt{\frac{(1-\cos \theta)e^t}{1-\cos \vartheta^0_t}} = \sqrt{1+\frac{2 (e^t-1)}{1-\cos \vartheta^0_t}}.
\end{equation}
Hence
\[
|(u^0_t)'(z_0)| \Lambda_t = \tan \frac{\theta}{2} \cot \frac{\vartheta^0_t}{2} - 1,
\]
\[
\frac{|(u^0_t)'(z_0)| \Lambda_t + |u^0_t(z_0)|}{|u_t^0(z_0) -1|} = \frac{|(u^0_t)'(z_0)|}{|u_t^0(z_0) -1|} = \frac{\tan \frac{\theta}{2} \cot \frac{\vartheta^0_t}{2}}{\sqrt{2(1-\cos \vartheta^0_t)}} 
\]
and
\begin{align*}
\int_0^t\frac{|(u^0_s)'(z_0)|}{|u_s^0(z_0) -1|^3} \dx s
&= 2^{-3/2}\tan \frac{\theta}{2} \int_0^t \frac{\cot \frac{\vartheta^0_s}{2}}{(1-\cos \vartheta^0_s)^{3/2}} \dx s \\
&\leq \frac{\tan \frac{\theta}{2}}{2^{3/2}\sqrt{1 + \cos \theta}} \int_0^t \frac{\cot \frac{\vartheta^0_s}{2} \sin \vartheta^0_s }{(1-\cos \vartheta^0_s)^{2}} \dx s \\
&=\frac{\tan \frac{\theta}{2}}{2^{3/2}\sqrt{1 + \cos \theta}} \frac{\cos \vartheta^0_t - \cos \theta}{(1-\cos \vartheta^0_t)(1-\cos \theta)} \\
&= \frac{e^{t/2}(1-e^{-t})\tan \frac{\theta}{2} \cot \frac{\vartheta^0_t}{2}}{2^{3/2}(1 - \cos \theta)\sqrt{1-\cos \vartheta^0_t}} \\
&\leq \frac{e^{t/2}\tan \frac{\theta}{2} \cot \frac{\vartheta^0_t}{2}}{4\sqrt{1-\cos \vartheta^0_t}},
\end{align*}
where we used the upper bound on $T$ in the final line.
The first result follows directly from Lemma \ref{ucomp}. For the second, as in the proof of Lemma \ref{ucomp},
\[
2 \left ( 1-\cos (\arg |\Psi_T(z)| - \xi_0) \right ) \leq |u^1_T(z) e^{-i \xi_0} - 1|^2 \leq 4|u^0_T(z_0) -1|^2 = 8 (1-\cos \vartheta^0_T), 
\]
and the result follows by using the assumption on $\|\xi\|_T$.
For the final result, observe that, by \eqref{radiusreverse} and Lemma \ref{ucomp} there exist absolute constants $A_i$ such that
\begin{align*}
& \ \log \frac{|\Psi_T(z)|-1}{|z|-1} \\ = & \ \int_0^T \frac{|u^1_t(z)|(|u^1_t(z)| + 1)}{|u^1_t(z)e^{-i \xi_{T-t}}-1|^2} \dx t \\
\geq & \ \int_0^T \frac{2}{|u^0_t(z_0)-1|^2} \dx t - \int_0^T \left | \frac{2}{||u^1_t(z)e^{-i \xi_{T-t}}-1|^2} - \frac{2}{|u^0_t(z_0)-1|^2}\right | \dx t\\
\geq & \ \int_0^T \frac{\dx t}{1-\cos \vartheta^0_t} - A_1 \int_0^T \frac{|u^1_t(z) - u^0_t(z_0)| + \| \xi \|_T}{|u^0_t(z_0)-1|^3}\dx t \\
\geq & \ \log \tan \frac{\theta}{2} \cot \frac{\vartheta^0_T}{2} \\
& \ - A_2 \left ( \| \xi\|_T \int_0^T \frac{\Lambda_t|(u_t^0)'(z_0)| + |u_t^0(z_0)|}{ |u_t^0(z_0)e^{-i \xi^0_{T-t}} -1|^3}\dx t + (|z|-1) \int_0^T \frac{|(u_t^0)'(z_0)| \dx t}{ |u_t^0(z_0)e^{-i \xi^0_{T-t}} -1|^3} \right ) \\
\geq & \ \log \tan \frac{\theta}{2} \cot \frac{\vartheta^0_T}{2} - A_3.
\end{align*}
Taking $A>e^{A_3}$ gives the required result.
\end{itemize}
\end{proof}

Next, we extend Corollary \ref{dertipetc} (ii) to give a lower bound on the derivative that holds for all values of $T$. 

\begin{lemma}
\label{phiprimebounds}
Suppose $\Psi_t(z)$ is the solution to the Loewner equation \eqref{loewnerPDEdrivingfunct}.
There exists some absolute constant $B$ such that, for all $T>0$ and $|z|>1$ satisfying
\[
\| \xi \|_T + |z|-1 \leq A^{-1}\sqrt{1-\cos (\arg z)},
\]
where $A$ is the absolute constant from Corollary \ref{dertipetc} (ii), we have
\[
|\Psi'_T(z)| \geq \frac{(|z|-1)\sqrt{1-\cos (\arg z)}}{B e^{T}(\|\xi\|_T+|z|-1)}. 
\]
\end{lemma}
\begin{proof}
We first obtain a generic lower bound on $|\Psi'_T(z)|$, without making any assumptions on the driving function $\xi$ or initial value $z$.
By \eqref{uderdef}
\[
\log |\Psi'_T(z)| \geq T-\int_0^T \frac{2}{|u^1_t(z)e^{-i \xi_{T-t}}-1|^2} \dx s = T-\int_0^T \frac{2 \partial_t r_t^1}{r_t^1\left ((r_t^1)^2-1\right )} \dx t = \log \frac{e^T |\Psi_T(z)|^2 (|z|^2-1)}{|z|^2(|\Psi_T(z)|^2-1)}.
\]
Therefore, using the fact that $|\Psi_T(z)|\geq |z|$,
\[
|\Psi'_T(z)| \geq \frac{e^T (|z|-1)}{|\Psi_T(z)|-1}.
\]
Now suppose $T$ satisfies the conditions of Corollary \ref{dertipetc} (ii). 
Then 
\[
|\Psi'_T(z)| \geq e^{-1} \tan(\arg(z)/2) \cot (\vartheta^0_T/2) \geq \frac{1}{3}
\]
and hence the required result holds provided $B \geq 3\sqrt{2}$.

If $T$ does not satisfy the conditions from Corollary \ref{dertipetc} (ii), then there exists some $0<S_1 < T$ such that 
\[
\| \xi \|_T + |z|-1 = A^{-1} e^{-S_1/2} \cot \frac{\arg z }{2} \tan \frac{\vartheta^0_{S_1}}{2}\sqrt{1-\cos \vartheta^0_{S_1}}.
\]
By \eqref{tancotid}, this is equivalent to 
\[
1-\cos \vartheta^0_{S_1} = A e^{S_1} (\| \xi \|_T + |z|-1)\sqrt{1-\cos (\arg z)}.
\]
We can write $\Psi_T(z) = \Psi_{T-S_1}(\psi_{S_1}(z))$ where $\psi_{S_1}$ is the solution to the Loewner equation for some driving function which is bounded by $\| \xi \|_T$. 
Using the generic estimate above, the results of Corollary \ref{dertipetc} (ii) applied to $\psi_{S_1}(z)$, the identity in \eqref{tancotid}, and that $| \Psi_T(z)| - 1 \leq 4|z| e^T$,
\begin{align*}
|\Psi'_T(z)| &\geq e^{T-S_1} \frac{|\psi_{S_1}(z)|-1}{|\Psi_T(z)|-1}|\psi'_{S_1}(z)| \\
&\geq \frac{e^{T-S_1} (|z|-1) \tan^2(\arg(z)/2) \cot^2(\vartheta^0_{S_1}/2)}{12 A |z| e^T} \\
&= \frac{(|z|-1) (1-\cos (\arg z))}{12A (1-\cos \vartheta^0_{S_1})} \\
&\geq \frac{(|z|-1) \sqrt{1-\cos (\arg z)}}{12 A^2 e^{T} (\| \xi \|_T + |z|-1)}.  
\end{align*}
Taking the absolute constant $B= 12 A^2$, gives the required result.
\end{proof}

Finally, we describe the radial and angular effect of the slit map $f_t(z)$ near the tip for
small values of $t$.

\begin{lemma}
\label{slitmapmove}
There exists some absolute constant $B$ such that, for all $0<t<1$ and $|z|>1$ with $|\arg z| \leq t^{1/2}$, we have 
\[
|f_t(z)| -1 \geq B^{-1} t^{1/2} \quad \mbox{and} \quad |\arg f_t(z)| \leq B(|z|-1).
\]
\end{lemma}
\begin{proof}
By \eqref{radiusreverse} and \eqref{anglereverse}, $|f_t(z)|$ is increasing in $t$ and $|\arg f_t(z)|$ is decreasing in $t$. Therefore, without loss, we may assume that $|z|-1 \leq A^{-1}t^{1/2}$ and $(1-\cos (\arg z))^{1/2} \geq A (|z|-1)$ where $A$ is the absolute constant from Corollary \ref{dertipetc} (ii). (Here we have used that $|\arg z| \asymp (1-\cos (\arg z))^{1/2}$). It follows that $|z|-1 \leq A^{-1}(1-\cos (\arg z))^{1/2}$ and so there exists some $s \leq \log (2/(1+\cos (\arg z)))$ such that 
\[
|z|-1 = A^{-1} e^{-s/2} \cot \frac{\arg z }{2} \tan \frac{\vartheta^0_s}{2}\sqrt{1-\cos \vartheta^0_s},
\]
where $\vartheta^0_s$ is defined as in \eqref{varthetadef} with $\theta=\arg z$. Observe that, by Corollary \ref{dertipetc} (ii),
\[
1 - \cos (\arg f_s(z)) \leq A (1-\cos \vartheta^0_s) = A^3 e^{s} \left ( (|z|-1) \tan \frac{\arg z }{2} \cot \frac{\vartheta^0_s}{2} \right )^2 \leq 3A^5 (|f_s(z)|-1)^2.
\]
Hence, again using \eqref{radiusreverse} and that $|\arg f_r(z)|$ is decreasing and $|f_r(z)|$ increasing in $r$, we have for all $s \leq r \leq t$ that
\begin{align*}
\partial_r|f_r(z)| &\geq |f_r(z)|\frac{|f_r(z)|^2-1}{(|f_r(z)|-1)^2+2|f_r(z)|(1-\cos (\arg f_s(z)))} \\
&\geq A_1^{-1} |f_r(z)|\frac{|f_r(z)|^2-1}{(|f_r(z)|-1)^2} 
\end{align*}
for some absolute constant $A_1$. It follows that 
\[
\log \left ( \frac{(|f_t(z)|+1)^2 }{4|f_t(z)|} \right ) \geq \log \left ( \frac{(|f_t(z)|+1)^2 |f_s(z)|}{|f_t(z)| (|f_s(z)|+1)^2} \right ) \geq \frac{t-s}{A_1} = \log \left ( \frac{(d((t-s)/A_1)+2)^2}{4 (d((t-s)/A_1)+1)} \right ) 
\]
and hence $|f_t(z)| \geq 1+ d((t-s)/A_1)$.
Since $0<t<1$, it is straightforward to verify that
\[
\log \frac{2}{1+\cos (t^{1/2})} \leq \frac{t}{2}
\]
and so $s \leq t/2$, Therefore $|f_t(z)|-1 \geq d(t/(2A_1)) \geq B_1^{-1}t^{1/2}$ for some absolute constant $B_1$.

By \eqref{radiusreverse} and \eqref{anglereverse},
\[
\frac{\partial_r (\arg f_r(z))}{\sin (\arg f_r(z))} = \frac{-2 \partial_r (|f_r(z)|)}{|f_r(z)|^2-1}
\]
and hence, integrating both sides,
\[
\tan \left (\frac{\arg f_t(z)}{2} \right ) = \tan \left ( \frac{\arg z}{2} \right ) \frac{(|f_t(z)|+1)(|z|-1)}{(|f_t(z)|-1)(|z|+1)} \leq B_1 \tan \left (\frac{t^{1/2}}{2} \right ) \frac{(2+d(t))(|z|-1)}{2 t^{1/2}}. 
\]
It follows that there exists some absolute constant $B \geq B_1$ such that $|\arg f_t(z)| \leq B(|z|-1)$.
\end{proof}

We are now in a position to return to the $\mathrm{ALE}(0,\eta)$ model and apply our results to prove Lemma \ref{phinbounds}, which we restate for convenience.

\begin{customlemma}{\ref{phinbounds}}
Fix $T>0$, let $n \leq \lfloor T/\capc \rfloor$ and set $\epsilon_n=(e^\parsig - 1) \vee \sup_{k\leq n} |\theta_k|$. 
\begin{itemize}
\item[(i)] 
There exists some absolute constant $A>1$, such that if $|\theta-\theta_n|<\capc^{1/2}$ and $\epsilon_n  < A^{-1} \capc^{1/2}$, then
\[
 \left| \left| \frac{\Phi'_n(e^{\parsig + i \theta})}{(f^{\theta_{n}}_{n \capc})'(e^{\parsig + i \theta})} \right | - 1 \right | < A \epsilon_n^2 \capc^{-1}.
\]
\item[(ii)] There exist absolute constants $A$ and $B$ only dependent on $T$, such that  if $\epsilon_n \leq A^{-1} \capc^{1/2}$, then 
\[
\left | \Phi'_n(e^{\parsig + i \theta}) \right | \geq B^{-1} \epsilon_n^{-1} \parsig (1 - \cos(\theta-\theta_n))^{1/2}.
\]
\end{itemize}
\end{customlemma}
\begin{proof}
\begin{itemize}
\item[(i)]
By the chain rule,
\[
\frac{\Phi'_n(e^{\parsig + i \theta})}{(f^{\theta_{n}}_{n \capc})'(e^{\parsig + i \theta})} = \frac{\Phi_{n-1}'(f^{\theta_n}_{\capc}(e^{\parsig + i \theta}))(f^{\theta_n}_{\capc})'(e^{\parsig + i \theta})}{(f^{\theta_n}_{(n-1)\capc})'(f^{\theta_n}_{\capc}(e^{\parsig + i \theta})) (f^{\theta_n}_{\capc})'(e^{\parsig + i \theta})} = \frac{\Phi_{n-1}'(f^{\theta_n}_{\capc}(e^{\parsig + i \theta}))}{(f^{\theta_n}_{(n-1)\capc})'(f^{\theta_n}_{\capc}(e^{\parsig + i \theta}))}.
\]
Set \[w=f^{\theta_n}_{\capc}(e^{\parsig + i\theta}) = e^{i \theta_n} f_\capc(e^{\parsig + i (\theta-\theta_n)}).\] Then if $|\theta-\theta_n| \leq  \capc^{1/2}$, by Lemma \ref{slitmapmove}, we have $|w|-1>B^{-1}\capc^{1/2}$ and $|\arg w - \theta_n| < B(e^{\parsig}-1)$ for some absolute constant $B$, and so 
\[
2 \epsilon_n + |\arg w - \theta_n| \leq (2B+B^2) \epsilon_n \capc^{-1/2} (|w|-1).
\] 
Since the conformal map $e^{i \theta_n}\Phi_{n-1}(z e^{-i \theta_n})$ has driving function bounded by $\sup_{k \leq n} |\theta_k - \theta_n| \leq 2 \epsilon_n$, by Corollary \ref{dertipetc} (i), there exists some constant $A$ (different to that in the corollary), such that if $\epsilon_n  < A^{-1} \capc^{1/2}$, then
\[
\left | \left | \frac{\Phi'_{n-1}(w)}{(f^{\theta_n}_{\capc(n-1)})'(w)} \right | -1 \right | \le A \epsilon_n^2 \capc^{-1}.
\]
Observe that it is not possible to apply Corollary \ref{dertipetc} (i) directly to $\Phi_n$ in the argument above, as this result requires $(|z|-1)/|\arg z - \theta_n|$ to be bounded away from zero which is not the case here. This is where we use that $\Phi_n$ evolves in discrete steps. Specifically, we invoke Lemma \ref{slitmapmove} to show that the single slit map $f_{\capc}^{\theta_n}$ maps $z$ into a region in which the condition needed for Corollary \ref{dertipetc} (i) holds. 

\item[(ii)] The result follows directly from Lemma \ref{phiprimebounds}.
\end{itemize}
\end{proof}

\section*{Acknowledgements}
We would like to thank the Isaac Newton Institute for Mathematical Sciences (Cambridge, UK) for support and hospitality during the programme ``Random Geometry'' where work on this paper was initiated. This work was supported by EPSRC grant no EP/K032208/1.

AS thanks the members of the Statistical Laboratory, University of Cambridge, for their hospitality during visits in the summers of 2016 and 2017. 

FV acknowledges generous support from the Knut and Alice Wallenberg Foundation, the Swedish Research Council, and the Gustafsson Foundation.

We are grateful to James Norris for numerous discussions and useful remarks, and for pointing out an error in an earlier version of our manuscript.  
Thanks go to Vittoria Silvestri for extensive comments on a draft of this paper. Finally, we thank two anonymous referees for a careful reading of our paper, and for incisive and detailed suggestions on how to improve the presentation.

\end{document}